\documentclass[11pt]{article}
\usepackage[a4paper,margin=1in]{geometry}
\usepackage[T1]{fontenc}
\usepackage[utf8]{inputenc}
\usepackage{lmodern} 
\usepackage{microtype}
\usepackage{amsmath,amsthm,amssymb,mathtools, mathrsfs}
\usepackage{bm}
\usepackage{enumitem}
\usepackage{hyperref}
\usepackage{xcolor}
\usepackage{centernot}
\usepackage{amsfonts}
\usepackage{natbib}  
\usepackage{comment}

\usepackage{setspace}

\newif\ifshowboxes
\showboxestrue % set to \showboxesfalse 

\ifshowboxes
  \usepackage{tcolorbox}
  
\else
  
\fi

\newtheorem{theorem}{Theorem}[section]
\newtheorem{proposition}[theorem]{Proposition}
\newtheorem{example}[theorem]{Example}
\newtheorem{corollary}[theorem]{Corollary}
\newtheorem{lemma}[theorem]{Lemma}
\theoremstyle{definition}
\newtheorem{definition}[theorem]{Definition}
\theoremstyle{definition}
\newtheorem{assumption}[theorem]{Assumption}
\theoremstyle{remark}
\newtheorem{remark}[theorem]{Remark}

\newcommand{\Mcal}{\mathcal{M}}
\newcommand{\Pcal}{\mathcal{P}}
\newcommand{\Qcal}{\mathcal{Q}}
\newcommand{\Ecal}{\mathcal{E}}
\newcommand{\E}{\mathbb{E}}
\newcommand{\R}{\mathbb{R}}
\newcommand{\N}{\mathbb{N}}
\newcommand{\1}{\mathbf{1}}
\newcommand{\Mone}{\mathcal{M}_1}
\newcommand{\Mplus}{\mathcal{M}_+}
\newcommand{\Bb}{\mathcal{B}_b}

\newcommand{\Peff}{\Pcal_\mathrm{eff}}
\newcommand{\co}{\operatorname{co}}
\newcommand{\cl}{\operatorname{cl}}
\newcommand{\sol}{\operatorname{sol}}
\newcommand{\ba}{\operatorname{ba}}

\newcommand{\GRO}{\operatorname{GRO}}
\newcommand{\GROW}{\mathsf{G}}
\newcommand{\REGROW}{\operatorname{REGROW}}

\newcommand{\Grow}{\mathsf{G}}
\newcommand{\Ms}{\mathcal{M}}

\newcommand{\sconv}{\sigma(\Ms,\Bb)}

\DeclareMathOperator{\conv}{co}
\DeclareMathOperator{\cconv}{\overline{\conv}}

\renewcommand{\P}{{\mathsf P}}
\renewcommand{\S}{{\mathsf S}}
\newcommand{\Q}{{\mathsf Q}}
\newcommand{\RR}{{\mathsf R}}

\title{Strong duality for the GROW criterion}
\author{
Ashwin Ram\footnote{Carnegie\ Mellon University, \texttt{aram2@andrew.cmu.edu}} 
\and
Martin Larsson\footnote{Department of Mathematical Sciences, Carnegie\ Mellon University, \texttt{larsson@cmu.edu}.} \and
Johannes Ruf\footnote{Department of Mathematics, London School of Economics, \texttt{j.ruf@lse.ac.uk}} \\[2ex]
\and
Aaditya Ramdas\footnote{Departments of Statistics and ML, Carnegie\ Mellon University, \texttt{aramdas@cmu.edu}} 
}
\date{\today}

\begin{document}

\maketitle

\begin{abstract}
This paper presents general strong duality results when testing hypotheses by betting against them. A bet is an e-variable for a composite null hypothesis $\Pcal$: a nonnegative random variable $X$ whose expected value is at most one under every $\P \in \Pcal$.
Following Kelly, Breiman, Cover, Shafer, and \cite{grunwald2024safe}, we study a natural minimax \emph{log-optimality} criterion: given a composite alternative $\Qcal$, we characterize the ``GROW value'' $\sup_{X} \inf_{\Q} \E_{\Q}[\log X]$. 
This paper generalizes the results of \cite{larsson2025numeraire} from (arbitrary $\Pcal$ and) simple $\Qcal$ to arbitrary $\Qcal$.
We prove that there always exists a minimizing information-projection pair between the weak-$*$ closures of the convex hulls of arbitrary $\Pcal$ and $\Qcal$, and show that the GROW value for \emph{bounded} e-variables always equals their relative entropy. We also prove a similarly general strong duality for the REGROW criterion with bounded e-variables and arbitrary bounded offsets. 
 Under various assumptions our results extend to unbounded e-variables, and examples show that without any assumptions such extensions fail. Our results are analogous to those in~\cite{larsson2026complete}, swapping tests for bounded e-variables, minimax risk for the GROW criterion, and total variation for relative entropy. 
\end{abstract}

% \begin{spacing}{0.2}
%     \tableofcontents
% \end{spacing}

\section{Introduction}
Given a measurable space $(\Omega,\mathcal{F})$, we write $\Mcal$ for the set of finite signed (countably additive) measures on $\mathcal{F}$, and $\Mcal_+$ and $\Mcal_1$ for the subsets of nonnegative measures and probability measures.
Let $\Pcal\subseteq\Mone$ be an arbitrary composite null hypothesis, and $\Qcal\subseteq\Mone$ an arbitrary composite alternative, both of which we always assume to be nonempty. 
An e-variable for $\Pcal$ is a $[0,\infty]$-valued random variable $X$ such that $\E_\P[X] \leq 1$ for all $\P \in \Pcal$.
The set of all e-variables for $\Pcal$ is denoted $\Ecal$, and the subset of e-variables bounded from above is denoted $\Ecal_b$. 

The influential work of~\cite{grunwald2024safe} proposed studying the \emph{GROW (Growth Rate that is Optimal in the worst-case) value}, 
\begin{equation}\label{eq:G}
\Grow=\sup_{X\in\mathcal{E}}\inf_{\Q\in\mathcal{Q}} \E_{\Q}[\log X],
\end{equation}
and in particular, establishing duality results that equate $\Grow$ to the infimum relative entropy between (appropriate extensions of) $\Qcal$ and $\Pcal$.
For example, \cite{larsson2025numeraire} point out that for simple $\Qcal=\{\Q\}$, $\Grow$ equals a particular infimum relative entropy between $\Q$ and the \emph{effective null hypothesis} corresponding to $\Pcal$. Similarly, for simple $\Pcal = \{\P\}$ and convex $\Qcal$,
Theorem~\ref{thm:singleton-P} yields, when (for example) $H(\Q \mid \P)<\infty$ for each $\Q \in \Qcal$, that $\Grow$ equals the infimum relative entropy between $\Qcal$ and $\P$.

When $\Pcal$ and $\Qcal$ are both composite, such a ``universal'' (i.e., without any additional restrictions) strong duality result does not hold without additional restrictions. For example, the results of \cite{Huber:Strassen:1973} imply that if $\Pcal$ and $\Qcal$ have a \emph{least favorable distribution} pair (a strong assumption!), then
\[
\Grow = \inf_{\P \in \Pcal} \inf_{\Q \in \Qcal} H(\Q \mid \P),
\]
where $H$ is the Kullback--Leibler divergence, or relative entropy;
see~\cite{saha2025huber} for a self-contained proof. 
\cite{grunwald2024safe} provide a different (also strong) sufficient condition under which the same statement holds. Our paper introduces several other conditions under which such characterizations hold; for example, under suitable compactness or domination assumptions on $\Pcal$ and $\Qcal$, many of which are weaker or incomparable to existing ones.

However, our investigations of the above quantity led us to a related notion that remarkably satisfies a strong duality without any restrictions on $\Pcal$ or $\Qcal$. In particular, we obtain a complete characterization of the \emph{bounded} GROW value, where the supremum is taken over the set $\Ecal_b$ of all \emph{bounded} e-variables,
\begin{equation}\label{eq:Gb}
\Grow_b=\sup_{X\in\mathcal{E}_b}\inf_{\Q\in\mathcal{Q}} \E_{\Q}[\log X].
\end{equation}
There are at least three reasons for considering bounded e-variables. First,~\cite{larsson2026complete} provide a complete characterization of testable hypotheses, that is, of when there exists a nontrivial test between $\Pcal$ and $\Qcal$; they point out this is equivalent to the existence of nontrivial \emph{bounded} e-variables. Second, as we formalize below, when $\Qcal$ is a singleton, we have $\Grow=\Grow_b$, so both~\eqref{eq:G} and~\eqref{eq:Gb} may equally be considered generalizations of the singleton case. Third, $\Grow_b$ permits a strong duality result without any restrictions, while this does not appear to be possible for $\Grow$.

The term GROW was originally motivated by the following betting context. We imagine a forecaster claiming that the data are well described by some $\P \in \Pcal$. A skeptic, believing instead that some $\Q \in \Qcal$ is a better description, would like to bet against the forecaster. The e-variable corresponds to an available bet or wealth multiplier, meaning that the skeptic puts forward (say) one dollar along with a particular e-variable before seeing the data, and the forecaster must pay back the realized value of the e-variable (in dollars) after the data is observed. The forecaster is willing to accept any e-variable as a bet, because under their own belief $\Pcal$, the Skeptic will have nonpositive expected net payoff. So which bet should the Skeptic choose? A long tradition dating back to \cite{kelly1956new,breiman1961optimal,cover1987log,shafer2021testing} suggests that the Skeptic should maximize the expected logarithm of the wealth under the alternative $\Q$. Picking the worst $\Q \in \Qcal$ results in the GROW criterion/value. Since $\E_{\Q}[\log X]$ is also sometimes called the e-power of $X$ under $\Q$ (see, for example,~\cite{ramdas2024hypothesis}), one can also refer to $\Grow$ as the minimax e-power, but we stick to the former terminology for simplicity.

As mentioned above, a universal strong duality result was recently established for singleton $\Qcal=\{\Q\}$ and arbitrary $\Pcal$ by~\cite{larsson2025numeraire}, extending earlier work by~\cite{harremoes2023universal} and~\cite{li1999estimation}.
Defining the \emph{effective null} generated by $\Pcal$ as
\begin{align} \label{eq:260523}
\Peff=\bigl\{\P\in\Mplus:\E_\P[X]\le 1\text{ for all } X\in\mathcal{E}\bigr\},
\end{align}
that paper showed that under no restrictions whatsoever on $\Q$ or $\Pcal$,
\begin{equation}\label{eq:numeraire-duality}
\sup_{X \in \Ecal} \E_\Q[\log X] =  \inf_{\P \in \Peff} H(\Q \mid \P),
\end{equation}
and identified a (sub-probability) measure $\P^* \in \Peff$ called the \emph{reverse information projection} such that the right-hand side above equals $H(\Q \mid \P^*)$.
One can view our work as generalizing these results to the setting of composite $\Pcal$ and $\Qcal$.

\subsection*{Summary of our main results}

To set the stage, we argue in Lemma~\ref{lem:E-to-Eb} that for a singleton $\Q$,
\begin{equation}\label{eq:e-eb-ebb-pointQ}
\sup_{X \in \Ecal} \E_\Q[\log X] =  \sup_{X \in \Ecal_b} \E_\Q[\log X].
\end{equation}
 So the left-hand side in~\eqref{eq:numeraire-duality} could also have used bounded e-variables. Further, by monotone convergence the right-hand side of~\eqref{eq:numeraire-duality}, which depends on $\Ecal$ through the definition of the effective null, could also have used bounded e-variables (see~\eqref{eq:260525}). Unsurprisingly, the story is more complicated for general composite $\Qcal$.

Obviously,  
\(
\Grow \geq \Grow_b,
\)
but the inequality can be strict --- remarkably, even for a point null $\Pcal = \{\P\}$ and composite $\Qcal$, both supported on the naturals $\N$ and thus having a common dominating measure; see Example~\ref{ex:example-from-ramdas-wang}. 
It turns out that $\Grow_b$ has a duality theorem without any restrictions. 

\begin{enumerate}
    \item 
Our main result, Theorem~\ref{T_strong_duality}, shows that \emph{under no assumptions on $\Pcal$ or $\Qcal$,} there always exists a pair $(\mu^*,\nu^*) \in \cconv^*(\Pcal)\times \cconv^*(\Qcal)$ such that
\begin{equation}\label{eq:intro-strong-duality-main}
\Grow_b =  \inf_{\mu \in \cconv^*(\Pcal)} \inf_{\nu \in \cconv^*(\Qcal)}  H(\nu \mid \mu) = H(\nu^* \mid \mu^*),
\end{equation}
where $\cconv^*(\Pcal)$ and $\cconv^*(\Qcal)$ are the closed convex hull of $\Pcal$ and $\Qcal$, respectively. Here we take the weak-$*$ closure in $\ba$, the space of bounded finitely additive measures. (We recall these concepts later.) One can also prove that for any $\Q \in \Mcal_1$,
\begin{equation}\label{eq:kl-equality-costar-peff}
\min_{\mu \in \cconv^*(\Pcal)} H(\Q \mid \mu) = \min_{\P \in \Peff} H(\Q \mid \P),
\end{equation}
so that \eqref{eq:intro-strong-duality-main}, in conjunction with~\eqref{eq:kl-equality-costar-peff} and~\eqref{eq:e-eb-ebb-pointQ}, implies~\eqref{eq:numeraire-duality}, justifying our claim that Theorem~\ref{T_strong_duality} generalizes the work of~\cite{larsson2025numeraire} to arbitrary $\Qcal$.

\item Following the suggestion of~\cite{grunwald2024safe}, we also study a generalization of GROW, called REGROW (relative GROW):
\[
\sup_{X\in\Ecal_b} \inf_{\Q \in \Qcal} \left( \E_\Q[\log X] - \xi(\Q) \right),
\]
for some bounded functional $\xi: \Qcal \to \R$, for example $\xi(\Q) = \sup_X \E_\Q[\log X]$ (the GROW criterion for a singleton $\Q$, called GRO), when this is finite and bounded over $\Qcal$. Theorem~\ref{thm:weak-star-offset-duality} establishes a strong duality result for arbitrary $\Pcal,\Qcal,\xi$ that involves the concave biconjugate of $\xi$.
\end{enumerate}

Next, we examine several interesting special cases where one can use \emph{unbounded} e-variables.

\begin{enumerate}
    \item[3.] First, if $\Pcal = \{\P\}$ is a singleton, and $\Qcal$ is convex with $H(\Q \mid \P) < \infty$ for all $\Q \in \Qcal$, we prove a strong duality result using a generalized information projection of $\P$ onto $\Qcal$:
\[
\Grow = \inf_{\Q \in \Qcal} H(\Q \mid \P).
\]
Example~\ref{ex:example-from-ramdas-wang} shows that despite strong duality holding for both $\Grow$ and $\Grow_b$, the two quantities are not equal in general. Example~\ref{ex:singletonP-continued} shows that if $H(\Q \mid \P) < \infty$ does not hold for all $\Q \in \Qcal$, then strong duality can fail, even though $\Q \ll \P$ for all $\Q \in \Qcal$. 

\item[4.] We generalize the above result to the case of arbitrary $\Pcal$, assuming that a joint information projection (JIPr) exists. A JIPr is defined as a pair $(\P^*,\Q^*)$ such that the sub-probability $\P^*$ is the reverse information projection of the probability $\Q^*$ onto $\Peff$ and $\Q^*$ is the generalized information projection (in the sense of~\cite{csiszar1984sanov}) of $\P^*$ onto $\Qcal$. When a JIPr exists, and (say) $H(\Q \mid \P^*) < \infty$ for all $\Q \in \Qcal$, we show that
\[
\Grow = \inf_{\Q \in \Qcal}\inf_{\P \in \Peff} H(\Q \mid \P),
\]
and the left-hand side is achieved by the e-variable $d\Q^*/d\P^*$.

\item[5.] Additionally, if $\Qcal$ is convex and compact in the setwise topology, then we show that
\[
\Grow = \Grow_b = \min_{\Q \in \Qcal} \min_{\P \in \Peff} H(\Q \mid \P).
\]
In the above two cases, if we further assume that (on Polish sample spaces) $\Pcal$ is convex and weakly compact, we show that $\Peff$ can be replaced by $\Pcal$. 
\item[6.] When working on Polish sample spaces, we show that if $\Pcal$ and $\Qcal$ are both convex and weakly compact, then again  
\[
\Grow = \Grow_b = \min_{\Q \in \Qcal} \min_{\P \in \Pcal} H(\Q \mid \P),
\]
but this time one can also further restrict to \textit{bounded continuous} e-variables.
\item[7.] Finally, we show that if $\Omega$ is finite, and $\Pcal,\Qcal$ are compact convex sets in the finite-dimensional probability simplex and if $\sup_{\P \in \Pcal} \P_i>0$ for each coordinate $i$,  then
\[
\Grow = \Grow_b = \min_{\Q \in \Qcal} \min_{\P \in \Pcal} H(\Q \mid \P).
\]
Further the GROW value is finite, attained by a bounded e-variable, which can be represented as the likelihood ratio of JIPr pairs that minimize the above right-hand side. 
\end{enumerate}

The first and sixth result above have an intriguing strong parallel to the main results of~\cite{larsson2026complete}, after swapping e-variables for tests, the GROW criterion for minimax risk, and the relative entropy for total variation distance.

We show a number of other results of independent interest,
along with several examples and counterexamples that demonstrate various subtleties in the aforementioned conditions.

The  paper is organized as follows. Section~\ref{sec:prelim} recaps some basic definitions and properties, especially of relative entropy when operating on finitely additive measures. Section~\ref{sec:main-results} presents  a strong duality for bounded GROW (and REGROW) without restriction on the sets of measures. Section~\ref{sec:grow} provides strong duality results for the unbounded GROW criterion, under some relatively mild restrictions. Section~\ref{sec:examples} discusses some interesting and nontrivial examples that demonstrate certain subtleties or challenges in these strong duality statements. The appendix contains all proofs not presented in the paper.

\section{Preliminaries}\label{sec:prelim}

We attempt to maintain the convention that for countably additive measures, we use Roman letters $\P,\Q,\RR,\S$, whereas for measures that could be finitely additive, we use Greek letters $\mu,\nu$.
Recall that the set of e-variables for $\Pcal$ is
\[
\Ecal = \left\{\text{all measurable } X \colon \Omega \to [0,\infty] \text{ such that } \int X d\P \le 1 \text{ for all } \P \in \Pcal\right\},
\]
and the set of all bounded e-variables is given by
\[
\Ecal_b = \left\{X \in \Ecal \colon \sup_{\omega \in \Omega} X(\omega) < \infty\right\} = \Ecal \cap \Bb,
\]
where $\Bb$ is the Banach space of all bounded measurable functions $f:\Omega\to \R$ with the supremum norm $\|f\|_\infty$.

Since $X = 1$ is an e-variable, all elements of $\Peff$, defined in \eqref{eq:260523}, are sub-probabilities. Any sub-probability which is setwise dominated by some element of $\Peff$ also belongs to $\Peff$. Moreover, $\Pcal$ and $\Peff$ have the same nullsets $A$ since if $\P(A) = 0$ for all $\P \in \Pcal$, then $\infty \1_A \in \Ecal$, and hence $\P(A) = 0$ for all $\P \in \Peff$. Finally, in the definition of $\Peff$ it is enough, thanks to the monotone convergence theorem, to let $X$ range over the set $\Ecal_b$. Hence we have
\begin{align} \label{eq:260525}
\Peff = 
\left\{\P \in \Mcal_+ \colon \int X d\P \le 1 \textnormal{ for all } X \in \Ecal\right\}
= \left\{\P \in \Mcal_+ \colon \int X d\P \le 1 \textnormal{ for all } X \in \Ecal_b\right\}.
\end{align}

We will also write $\ba=\ba(\Omega,\mathcal F)$ for the Banach space of all bounded finitely additive signed measures on $(\Omega,\mathcal F)$, endowed with the total variation norm. Its positive cone is denoted by $\ba_+$. Moreover, $\ba_1=\{\mu\in\ba_+:\mu(\Omega)=1\}$ denotes the set of all finitely additive probability measures on $(\Omega,\mathcal F)$. Each $\mu\in\ba_1$ acts on $\Bb$ through finitely additive integration; for $f\in\Bb$, we write $\E_\mu[f]=\int f\,d\mu$. Thus, $\ba_1$ may equivalently be viewed as the positive normalized part of the dual of $(\Bb,\|\cdot\|_\infty)$.
For $\mu\in\ba_1$, we will always let $\mu = \mu_c + \mu_p$ denote its Yosida-Hewitt decomposition, where $\mu_c\in\Mplus$ is countably additive and $\mu_p$ is purely finitely additive. 
Moreover, $\cconv^*$ denotes the weak-$*$ closed convex hull, that is, for a set $\mathcal A\subseteq\ba$, $\cconv^*(\mathcal A)$ denotes the closure of the convex hull of $\mathcal A$, taken in the space $\ba$ with the topology $\sigma(\ba, \Bb)$.

\begin{remark}\label{rem:setwise-weak-tv}
Unless specified otherwise, no topology is imposed on $\Omega$; we work on a general measurable
space $(\Omega,\mathcal F)$. If $\Omega$ is a topological space and
$\mathcal F$ is its Borel sigma algebra, then the usual weak topology on
$\mathcal M$ is $\sigma(\mathcal M,C_b)$, where $C_b$ denotes the family of bounded
continuous functions. This topology is coarser than the setwise topology
$\sigma(\mathcal M,\Bb)$ since $C_b\subseteq\Bb$. Thus setwise convergence
implies weak convergence, but not conversely; for example, on $\mathbb R$,
$\delta_{1/n}$ converges weakly to $\delta_0$, but not setwise.
On the other hand, the total variation topology is finer than
$\sigma(\mathcal M,\Bb)$. Indeed,
\[
\|\RR\|_{\mathrm{TV}}
=
\sup_{\|f\|_\infty\le1}
\left|\int f\,d\RR\right|,
\]
so total variation convergence is uniform convergence over the unit ball of
$(\Bb,\|\cdot\|_\infty)$, whereas convergence in $\sigma(\mathcal M,\Bb)$ is
only pointwise convergence against each fixed $f\in\Bb$.
\end{remark}

For a nonnegative measurable function $f: \Omega\to [0,\infty]$, and $\mu\in\ba_1$, we define the extended finitely additive integral as $\E_\mu[f]=\sup_{n\in \N}\E_{\mu}[f\wedge n]\in[0,\infty]$; equivalently, $\E_{\mu}[f]=\sup\{\E_{\mu}[g]:g\in\Bb, 0\le g\le f\}$. Clearly, this definition coincides with the familiar Lebesgue integral whenever $\mu\in\Mone\subseteq\ba_1$. Let us adopt the conventions $\log 0=-\infty$ and $\log \infty = \infty$. 
Whenever $\E_{\mu}[(\log f)^{-}]<\infty$, we set $\E_{\mu}[\log f] = \E_{\mu}[(\log f)^+] - \E_{\mu}[(\log f)^-]\in(-\infty,\infty]$. Whenever $\E_{\mu}[(\log f)^-]=\infty$, we define $\E_{\mu}[\log f]=-\infty$. Thus, by convention, this quantity equals $-\infty$ whenever the negative logarithmic part has infinite finitely additive integral.

The following convention is used when we evaluate the expression $q \log(q/p)$ for $q \in [0,1]$ and $p\geq0$:
\begin{equation} \label{eq_q_log_q_by_p_convention}
q \log\frac{q}{p} = \begin{cases} 0, & q = 0, \\ \infty, & p = 0 \text{ and } q>0.
\end{cases}
\end{equation}
This ensures that $q \log(q/p)$ is jointly convex and lower semicontinuous in $(p,q)$.

\begin{definition}[Relative entropy over $\ba_1$]\label{def:H}
We write $\Pi$ for the set of all finite measurable partitions of $\Omega$. For any $\nu \in \ba_1$,  $\mu \in \ba_+$, and finite measurable partition $\pi = \{A_1,\ldots,A_n\}$ we define, using the convention \eqref{eq_q_log_q_by_p_convention},
\[
H_\pi(\nu \mid \mu) = \sum_{i=1}^n \nu(A_i) \log \frac{\nu(A_i)}{\mu(A_i)}.
\]
Their Kullback--Leibler divergence or relative entropy is then given by
\[
H(\nu \mid \mu) = \sup_{\pi \in \Pi} H_\pi(\nu \mid \mu).
\]
\end{definition}

For $\Q\in\Mone$ and $\P\in\Mplus$, it can be shown that Definition~\ref{def:H} reduces to 
\[
H(\Q\mid \P)=
\begin{cases}
\displaystyle \int\log\left(\frac{d\Q}{d\P}\right)\,d\Q, &\Q\ll \P,\\
\infty, &\Q\not\ll \P,
\end{cases}
\]
which we note can take negative values when $\P(\Omega) > 1$.

Even for finitely additive probability measures, zero relative entropy implies equality. 

\begin{lemma}\label{lem:weak-star-jipr-zero-iff-equal}
For $\mu,\nu\in\ba_1$, $H(\nu \mid \mu) \geq 0$ with equality if and only if $\mu=\nu$.    
\end{lemma}

All proofs for the statements in this section are deferred to Appendix~\ref{sec:missing-proofs}.

There is also a Donsker--Varadhan dual representation of relative entropy.  
\begin{lemma}\label{lem:extended-entropy-sec}
For $\nu \in \ba_1$ and $\mu \in \ba_+$, we have
\[
H(\nu\mid\mu)=\sup_{f\in\Bb}\left(\int fd\nu-\log\int e^f d\mu\right).
\]

Moreover, for every measurable
\(f:\Omega\to[-\infty,\infty]\) such that
\(0<\int e^f\,d\mu<\infty\), we have
\[
\int f\,d\nu-\log\int e^f\,d\mu\leq H(\nu\mid \mu).
\]
\end{lemma}

\begin{proposition}\label{prop:ba-countably-additive-reduction}
    For $\Q \in \Mcal_1$ and $\mu \in \ba_+$, we have
    \(
    H(\Q \mid \mu) = H(\Q \mid \mu_c).
    \)
For $\nu\in\ba_{1}\setminus\Mone$ and $\P\in\Mplus$, we have $H(\nu\mid \P)=\infty$. Thus, for $\nu \in \ba_1$ and $\P\in\Mplus$, $H(\nu \mid \P) < \infty$ implies that $\nu \in \Mone$ and $\nu \ll \P$.
\end{proposition}

Last, we note that for simple alternatives, it is enough to optimize over \emph{bounded} e-variables.

\begin{lemma}\label{lem:E-to-Eb}
    For any $\Q \in \Mone$,
   $\sup_{X \in \Ecal}\E_\Q[\log X] = \sup_{X \in \Ecal_b} \E_\Q[\log X]$. 
\end{lemma}

Note however that the supremum on the left-hand side is achieved (by the numeraire), but on the right-hand side in general it is not.

\section{Strong duality for  bounded GROW and REGROW}\label{sec:main-results}
This section will present and prove the main strong duality result.
The reader may find the following weak duality useful for intuition for the results that follow:
\begin{align} 
\sup_{X \in \Ecal} \inf_{\Q \in \Qcal} \E_\Q[\log X] &= \sup_{X \in \Ecal} \inf_{\Q \in \conv(\Qcal)} \E_\Q[\log X] \leq \inf_{\Q \in \conv(\Qcal)} \sup_{X \in \Ecal} \E_\Q[\log X] \nonumber \\ &
=
\inf_{\Q \in \conv(\Qcal)} \inf_{\P \in \Peff}
H(\Q \mid \P) 
\leq 
\inf_{\Q \in \conv(\Qcal)} \inf_{\P \in \conv(\Pcal)}
 H(\Q \mid \P),  \label{weak_duality_general}
\end{align}
and note that $\Ecal$ could have been replaced by $\Ecal_b$ above. Our aim is to establish when \emph{strong duality} holds, i.e., when we have equality in suitably adjusted versions of the above display.

\begin{theorem} \label{T_strong_duality}
The following strong duality holds:
\begin{equation} \label{eq_strong_duality_general}
\sup_{X \in \Ecal_b} \inf_{\Q \in \Qcal} \E_\Q[\log X] = 
\min_{\nu \in \cconv^*(\Qcal)} 
\min_{\mu \in \cconv^*(\Pcal)}
 H(\nu \mid \mu),
\end{equation}
even if one side, and then also the other side, equals infinity. Further, there exists a pair $(\mu^*,\nu^*) \in \cconv^*(\Pcal)\times\cconv^*(\Qcal)$ that achieves the minimum on the right-hand side.
\end{theorem}

This theorem follows from Theorem~\ref{thm:weak-star-offset-duality} by setting $\xi=0$ therein (in which case  $\xi^{**}_\Qcal=0$ on $\cconv^*(\Qcal)$).

One may ask whether considering the effective null or closures in $\ba$ is really necessary.
Here we present an example where 
$\Qcal$ is a singleton, $\Pcal$ is convex,  
$\inf_{\P \in \Pcal} H(\Q \mid \P)$ is achieved in $\Pcal$ and yet this value is vastly different from $\inf_{\mu \in \cconv^*(\Pcal)} H(\Q \mid \mu)$ and from $\inf_{\P \in \Peff} H(\Q \mid \P)$.

\begin{example}
    Let $\Omega = [0,1]$. Let $\Q = \mathsf U$ be the uniform distribution. Let $\Pcal$ be the convex hull of $\P_0$ and all $\{\delta_z:z\in[0,1]\}$, where $\delta_z$ denotes a Dirac delta mass at $z$ and $\P_0 = \varepsilon \mathsf U + (1-\varepsilon)\delta_0$ for some small constant $\varepsilon > 0$.  Clearly, $\P_0$ achieves the infimum $\inf_{\P \in \Pcal} H(\Q \mid \P)$, which can be made arbitrarily large by letting $\varepsilon$ tend to zero. However, $\Peff$ contains all distributions on $[0,1]$, and thus contains $\Q$, causing $\inf_{\P \in \Peff} H(\Q \mid \P) = 0$. \citet[Theorem 3.3]{larsson2026complete} shows that $\Peff \cap \Mone=\cconv^*(\Pcal) \cap \Mone$ and so $\Q \in \cconv^*(\Pcal)$, yielding $\inf_{\mu \in \cconv^*(\Pcal)} H(\Q \mid \mu) = 0$.
\end{example}

The following example shows that every minimizing pair $(\mu^*,\nu^*)$ on the right-hand side of \eqref{eq_strong_duality_general} can in fact be purely finitely additive.  

\begin{example}\label{ex:weak-star-jipr-necessarily-finitely-additive}
Let $\Omega=\mathbb N$ be equipped with its power set sigma algebra.
 For each $n\in\mathbb N$,  define
\[
\P_n=\frac1n \sum_{k=1}^n \delta_k,\qquad \Q_n=\frac1n \sum_{k=2}^{n+1} \delta_k
\]
and set $\Pcal=\{\P_n:n\in\N\}$ and $\Qcal=\{\Q_n:n\in\N\}$.
Then the right-hand side of \eqref{eq_strong_duality_general} is equal to zero. Moreover, every minimizing pair $ (\mu^*, \nu^*) \in \cconv^*(\Pcal) \times \cconv^*(\Qcal)$ satisfies $ \mu^*=\nu^* \in
\cconv^*(\Pcal) \cap \cconv^*(\Qcal)$.
Every such common minimizer is purely finitely additive. In particular, there is no minimizing pair with either $\mu^*\in\Mone$ or $\nu^*\in\Mone$. These claims are argued in Appendix~\ref{app:example}.
\end{example}

 \cite{grunwald2024safe} point out that $\GROW$ can be too pessimistic as a criterion, as it is a worst case over all alternatives, and thus any (approximately) GROW e-variable may achieve poor e-power against an ``easy'' $\Q$ in order to preserve optimal e-power against a worst-case $\Q$. To address this, they suggest the $\REGROW$ criterion, which normalizes the objective by the best achievable growth rate for that alternative. In this sense, a REGROW e-variable attempts to achieve (as much as possible) nearly optimal e-power for every $\Q \in \Qcal$, and thus is a possibly more appropriate criterion to work with, motivating us to present the following theorem (for general bounded offsets $\xi$).

\begin{theorem}\label{thm:weak-star-offset-duality}
Let $\xi:\Qcal\to\R$ be a bounded function, and
define its concave biconjugate $\xi^{**}_\Qcal$  for $\nu \in \ba_+$ by 
\begin{equation} \label{eq:260531.1}
\xi_{\Qcal}^{**}(\nu)=\inf_{f\in\Bb}\left(\int f\,d\nu - \inf_{\Q\in\Qcal}\left(\E_\Q[f]-\xi(\Q)\right)\right).
\end{equation}
Then the following strong duality holds:
\begin{align}
\sup_{X\in\Ecal_{b}}\inf_{\Q\in\Qcal}(\E_\Q[\log X]-\xi(\Q)) =
\min_{\nu \in \cconv^*(\Qcal)}
\min_{\mu \in \cconv^*(\Pcal)} 
\left(H(\nu\mid\mu)-\xi^{**}_{\Qcal}(\nu)\right),\label{eq:weak-star-offset-form}
\end{align}
even if one side, and then also the other side, equals infinity. Further, there exists a pair $(\mu^*,\nu^*) \in \cconv^*(\Pcal)\times\cconv^*(\Qcal)$ that achieves the minimum on the right-hand side.
\end{theorem}
Before proving this theorem we provide a remark and some corollaries and auxiliary lemmata.

\begin{remark}
    Since the offset function $\xi$ will typically be convex rather than concave, it cannot be used directly in the minimax argument. We therefore work with its concave biconjugate, given by \eqref{eq:260531.1}.
Let us write, for $\Q \in \Qcal$,  
\[
\Grow(\Q)=\Grow_b(\Q)=\sup_{X\in\Ecal_b}\int\log X\,d\Q.
\]
For $\Qcal = \{\Q\}$, this agrees both with the GROW value $\Grow$ and the bounded GROW value $\Grow_b$ in \eqref{eq:G} and \eqref{eq:Gb}, respectively.
It is the maximal e-power against the simple alternative $\nu$ and coincides with the $\GRO$ value in the terminology of \cite{grunwald2024safe}. 
Let us now assume $$\sup_{\Q\in\Qcal}\Grow(\Q)<\infty$$ and set $\xi =\Grow$. Then, the value in \eqref{eq:weak-star-offset-form} presents REGROW (for bounded e-variables):
\begin{equation*}
\sup_{X\in\Ecal_{b}}\inf_{\Q\in\Qcal}(\E_\Q[\log X]-\Grow(\Q)) =
\min_{\nu \in \cconv^*(\Qcal)}
\min_{\mu \in \cconv^*(\Pcal)}
\left(H(\nu\mid\mu)-\Grow^{**}_\Qcal(\nu)\right) = 
\min_{\nu \in \cconv^*(\Qcal)}  (\Grow(\nu) -  \Grow^{**}_\Qcal(\nu)),
\end{equation*}
where the final equality follows by Lemma~\ref{lem:generalization-of-numeraire}, presented below.
\end{remark}

\begin{corollary} \label{C:260526}
    For any $\Q \in \Mcal_1$,
    \[
    \min_{\P \in \Peff} H(\Q \mid \P) = \min_{\mu \in \cconv^*(\Pcal)} H(\Q \mid \mu).
    \]
\end{corollary}
\begin{proof}
    We get the chain of equalities
    \[
    \min_{\P \in \Peff} H(\Q \mid \P) =
    \sup_{X \in \Ecal} \E_\Q[\log X] = 
    \sup_{X \in \Ecal_b} \E_\Q[\log X] = 
    \min_{\mu \in \cconv^*(\Pcal)} H(\Q \mid \mu),
    \]
    where the first equality follows from~\cite{larsson2025numeraire}, the second from Lemma~\ref{lem:E-to-Eb}, 
    and the last by taking $\Qcal=\{\Q\}$ in \eqref{eq_strong_duality_general}.
\end{proof}

There is also an interesting corollary regarding the existence of a \emph{nontrivial} test, which we define as any test whose worst-case power exceeds its worst-case level. To set the stage, a test is a measurable function $\phi$ whose range is $[0,1]$, its worst-case type-I error is $\sup_{\P \in \Pcal} \E_\P[\phi]$ and its worst-case power is $\inf_{\Q \in \Qcal} \E_\Q[\phi]$. A classical theorem by \cite{kraft1955some} (credited also to Le Cam) asserts that \emph{if $\Pcal \cup \Qcal$ are dominated by a common reference measure}, then a nontrivial test exists if and only if $\conv(\Pcal)$ and $\conv(\Qcal)$ are separated in the total variation distance. Recently,~\cite{larsson2026complete} proved that the aforementioned reference measure assumption can be dropped if $\conv(\Pcal)$ and $\conv(\Qcal)$ are replaced by $\cconv^*(\Pcal)$ and $\cconv^*(\Qcal)$, respectively. Our theorem above delivers the same corollary.

\begin{corollary}
    A nontrivial test for $\Pcal$ against $\Qcal$ exists if and only if $d_{\mathrm{TV}}(\cconv^*(\Pcal),\cconv^*(\Qcal))>0$.
\end{corollary}
\begin{proof}
 The sets $\cconv^*(\Pcal)$ and $\cconv^*(\Qcal)$ are weak-$*$ compact, and the map $(\mu, \nu)\mapsto \|\mu-\nu\|_{\mathrm{TV}}$ is weak-$*$ lower semicontinuous. Hence the total variation distance between the two sets is attained. Therefore this distance is positive if and only if $\cconv^*(\Pcal)$ and $\cconv^*(\Qcal)$ are disjoint.
 
 By \cite{ramdas2024hypothesis}, a nontrivial test exists if and only if there exists $X\in\Ecal_b$ 
 such that $\inf_{\Q \in \Qcal} \E_\Q[\log X] > 0$.
 Equivalently, the left-hand side of \eqref{eq_strong_duality_general} is strictly positive, and thus so is the right-hand side. Since relative entropy 
 between two measures in $\ba_1$ equals zero if and only they are equal by Lemma~\ref{lem:weak-star-jipr-zero-iff-equal}, this is equivalent to $\cconv^*(\Pcal)$ and $\cconv^*(\Qcal)$ being disjoint, and this concludes the proof.
\end{proof}

Before proving Theorem~\ref{thm:weak-star-offset-duality}, we establish a few facts concerning the concave biconjugate.

\begin{lemma}\label{lem:offset-transform}
For $\mathcal A \subseteq \ba_+$ nonempty, suppose $\xi:\mathcal A\to\R$ is bounded and 
define $\xi^{**}_\mathcal A$ as in \eqref{eq:260531.1}.
Then $\xi^{**}_{\mathcal A}$ is weak-$*$-upper semicontinuous and concave on $\cconv^*(\mathcal A)$, and  for every $g\in\Bb$, \begin{equation}\label{eq:offset-transform-inf-preserving} \inf_{\nu\in\mathcal A}\left(\int g\,d\nu-\xi(\nu)\right) = \inf_{\nu\in \cconv^*(\mathcal A)} \left( \int g\,d\nu - \xi^{**}_{\mathcal A}(\nu) \right). \end{equation}
\end{lemma}

\begin{proof}
For $g\in\Bb$, define the concave conjugate of $\xi$ by
\[
a_{\xi}(g)=\inf_{\nu\in\mathcal A}\left(\int g\,d\nu-\xi(\nu)\right)
\]
and set
\[
\psi_g(\nu)=\int g\,d\nu-a_{\xi}(g),
\qquad \nu\in\cconv^*(\mathcal A).
\]
Then $\psi_g$ is affine and weak-$*$ continuous on
$\cconv^*(\mathcal A)$ and we have
$
\xi^{**}_{\mathcal A}=\inf_{g\in\Bb}\psi_g$.
The pointwise infimum of continuous affine functions is upper semicontinuous
and concave. Hence $\xi^{**}_{\mathcal A}$ is
weak-$*$ upper semicontinuous and concave on $\cconv^*(\mathcal A)$.

It remains to prove \eqref{eq:offset-transform-inf-preserving}. Fix $g\in\Bb$. From the definition of $\xi^{**}_{\mathcal A}$,
\begin{equation} \label{eq:260530}
a_{\xi}(g) \leq 
\inf_{\nu\in\cconv^*(\mathcal A)}
\left(
\int g\,d\nu-\xi^{**}_{\mathcal A}(\nu)
\right).
\end{equation}
For the reverse inequality, first note that for every $\nu_0\in\mathcal A$ and every $h\in\Bb$,
\[
a_\xi(h)
=
\inf_{\nu\in\mathcal A}\left(\int h\,d\nu-\xi(\nu)\right)
\le
\int h\,d\nu_0-\xi(\nu_0).
\]
Thus
\[
\int h\,d\nu_0-a_\xi(h)\ge \xi(\nu_0).
\]
Taking the infimum over $h\in\Bb$ gives $
\xi^{**}_{\mathcal A}(\nu_0)\ge \xi(\nu_0)$.
Consequently,
\begin{align*}
\inf_{\nu\in\cconv^*(\mathcal A)}
\left(
\int g\,d\nu-\xi^{**}_{\mathcal A}(\nu)
\right)
\le
\inf_{\nu\in\mathcal A}
\left(
\int g\,d\nu-\xi^{**}_{\mathcal A}(\nu)
\right)  \le
\inf_{\nu\in\mathcal A}
\left(
\int g\,d\nu-\xi(\nu)
\right) =
a_\xi(g).
\end{align*}
Combining this with \eqref{eq:260530} proves \eqref{eq:offset-transform-inf-preserving}.
\end{proof}

The following short lemma will be useful below.
\begin{lemma} \label{L:260531}
Let $\mathcal A \subseteq\ba_1$ be nonempty. Let $\xi:\mathcal A\to\R$ be a bounded function, 
and define
\begin{align}
\label{eq:Gamma}
\Gamma
=
\left\{
f\in\Bb:
\sup_{\P\in\Pcal}\E_\P[e^f]\le 1
\right\}.
\end{align}
Then $\Gamma$ is convex and
\[
\sup_{X\in\Ecal_{b}}\inf_{\nu\in\mathcal A}(\E_\nu[\log X]-\xi(\nu))
= \sup_{f\in\Gamma}\inf_{\nu\in\mathcal A}(\E_\nu[f]-\xi(\nu)).
\]
\end{lemma}
\begin{proof}
We first observe that $\Gamma$ is convex due to H\"older's inequality. The inequality ``$\ge$'' is clear. For the reverse inequality, fix $X\in\Ecal_b$ and $\varepsilon\in(0,1)$, and set
$
X_\varepsilon=\varepsilon+(1-\varepsilon)X \in \Ecal_b$.
Then $f_\varepsilon = \log X_\varepsilon\in\Gamma$. Since $X_\varepsilon \geq (1-\varepsilon)X$, we have $f_\varepsilon \geq \log (1-\varepsilon) + \log X$. Therefore
\[
\inf_{\nu \in \mathcal A}\left(\int f_\varepsilon d \nu - \xi(\nu)\right) \geq \inf_{\nu \in \mathcal A}\left(\int \log X d \nu - \xi(\nu)\right) + \log(1-\varepsilon).
\]
Sending $\varepsilon$ to zero and taking the supremum over $X \in \Ecal_b$ proves the reverse inequality.
\end{proof}

In preparation for the proof of Theorem~\ref{thm:weak-star-offset-duality}, we also provide a simple-alternative duality.
The main differences between this result and~\eqref{eq:numeraire-duality} are that  $\nu$ is allowed to be in $\ba_1$ (instead of just $\Mcal_1$), and the right-hand side minimizes over $\cconv^*(\Pcal)$ instead of $\Peff$.

\begin{lemma}\label{lem:generalization-of-numeraire}
Let $\Gamma$ be as in \eqref{eq:Gamma}.
For $\nu\in\ba_1$, 
\begin{equation*}
\sup_{X\in \Ecal_b}\int \log X\,d\nu
=
\sup_{f\in \Gamma}\int f\,d\nu
=\min_{\mu\in\cconv^*(\Pcal)} H(\nu\mid\mu).
\end{equation*}
\end{lemma}

\begin{proof}
By Lemma~\ref{L:260531}, applied with $\xi\equiv 0$, it suffices to argue
\begin{equation}\label{eq:single-measure-duality2}
\sup_{f\in\Gamma}\int f\,d\nu=\min_{\mu\in\cconv^*(\Pcal)} H(\nu\mid\mu).
\end{equation}

To make headway, for $\mu\in\cconv^*(\Pcal)$ and
$f\in\Bb$, define
\[
\Phi(f,\mu)
=
\int f\,d\nu-\log\int e^f\,d\mu .
\]
For each fixed $\mu\in\cconv^*(\Pcal)$, the map
$f\mapsto \Phi(f,\mu)$ is concave on  $\Bb$. Indeed,
$f\mapsto \int f\,d\nu$ is affine, while
$f\mapsto \log\int e^f\,d\mu$ is convex by H\"older's inequality. Moreover, this map is continuous
with respect to the supremum norm. For each fixed $f\in\Bb$, the map
$\mu\mapsto \Phi(f,\mu)$ is convex and continuous on
$\cconv^*(\Pcal)$, since $\mu\mapsto\int e^f\,d\mu$ is affine and
continuous, takes values in $(0,\infty)$, and
$x\mapsto -\log x$ is continuous and convex on $(0,\infty)$.

Since $\cconv^*(\Pcal)$ is compact by the Banach--Alaoglu theorem and convex, and $\Bb$ is convex,
Sion's minimax theorem applies. Together with the variational formula for
relative entropy (Lemma~\ref{lem:extended-entropy-sec}), this yields
\begin{align}
\min_{\mu\in\cconv^*(\Pcal)} H(\nu\mid\mu)
&=
\min_{\mu\in\cconv^*(\Pcal)}
\sup_{f\in\Bb}\Phi(f,\mu) =
\sup_{f\in\Bb}
\min_{\mu\in\cconv^*(\Pcal)}\Phi(f,\mu) \notag\\
&=
\sup_{f\in\Bb}
\left(
\int f\,d\nu
-
\sup_{\mu\in\cconv^*(\Pcal)}
\log\int e^f\,d\mu
\right).
\label{eq:single-measure-after-sion}
\end{align}

For $f\in\Bb$ set $c(f)=\sup_{\mu\in\cconv^*(\Pcal)}\log\int e^f\,d\mu$. Then we have $f-c(f)\in\Gamma$, yielding
\[
\int f\,d\nu-c(f)=\int (f-c(f))\,d\nu\le \sup_{g\in\Gamma}\int g\, d\nu.
\]
From 
\eqref{eq:single-measure-after-sion} we hence get
\begin{equation}\label{eq:single-measure-upper}
\min_{\mu\in\cconv^*(\Pcal)} H(\nu\mid\mu)\le \sup_{g\in\Gamma}\int g\,d\nu. 
\end{equation}
Conversely, suppose that $g\in\Gamma$. We know that $c(g)\le 0$. Hence  \eqref{eq:single-measure-after-sion} yields
\[
\int g\,d\nu\le \int g\,d\nu-c(g)\le \sup_{f\in\Bb}\left(\int f\, d\nu-c(f)\right)=\min_{ \mu\in\cconv^*(\Pcal)}H(\nu\mid\mu).
\]
Together with \eqref{eq:single-measure-upper}, this yields \eqref{eq:single-measure-duality2}, concluding the proof.
\end{proof}

\begin{proof}[Proof of Theorem~\ref{thm:weak-star-offset-duality}]
Let $\Gamma$ be as in \eqref{eq:Gamma}.
We first note 
\[
\sup_{X\in\Ecal_{b}}\inf_{\Q\in\Qcal}(\E_\Q[\log X]-\xi(\Q))
= \sup_{f\in\Gamma}\inf_{\Q\in\Qcal}(\E_\Q[f]-\xi(\Q))
=
 \sup_{f\in\Gamma} \inf_{\nu\in \cconv^*(\Qcal)} \left( \int f\,d\nu - \xi^{**}_{\Qcal}(\nu) \right)
\]
by Lemmata~\ref{L:260531} and \ref{lem:offset-transform}.
We now intend to apply Sion's minimax theorem. To this end, recall that $\Gamma$ is convex  by Lemma~\ref{L:260531} and $\cconv^*(\Qcal)$ is compact by the Banach--Alaoglu theorem, respectively. 
We next define 
\[
L(f,\nu)=\int fd\nu-\xi^{**}_{\Qcal}(\nu).
\]
By Lemma~\ref{lem:offset-transform}, $\nu \mapsto L(f,\nu)$ is lower semicontinuous and convex on $\cconv^*(\Qcal)$. Moreover, $f \mapsto L(f,\nu)$ is linear and continuous. 
Further, $\xi^{**}_{\Qcal}(\nu) \in [\inf_{\Q \in \Qcal} \xi(\Q), \sup_{\Q \in \Qcal} \xi(\Q)]$, and thus $\xi^{**}_\Qcal$ is real-valued on $\cconv^*(\Qcal)$.
Therefore, Sion's minimax theorem applies and we obtain
\begin{align*}
\sup_{f\in\Gamma} \inf_{\nu\in \cconv^*(\Qcal)} \left( \int f\,d\nu - \xi^{**}_{\Qcal}(\nu) \right)
= \inf_{\nu\in \cconv^*(\Qcal)} \sup_{f\in\Gamma}  \left( \int f\,d\nu - \xi^{**}_{\Qcal}(\nu) \right).
\end{align*}
By Lemma~\ref{lem:generalization-of-numeraire}, we get \eqref{eq:weak-star-offset-form}.
Finally, attainment follows from lower semicontinuity and compactness of 
$\cconv^*(\Pcal) \times \cconv^*(\Qcal)$.
\end{proof}

To provide further intuition for the concave biconjugate, we discuss the simpler case of a finite full simplex.

\begin{example}
Suppose that $\Omega=\{1,\ldots,n\}$, and let $\P\in\Delta_n$ have full support, that is,
$
p_i=\P(\{i\})>0$ for every $i=1,\ldots,n$.
Let $\Pcal=\{\P\}$ and $\Qcal=\Delta_n$. For $\Q\in\Qcal$, we set again
\[
\Grow(\Q)=H(\Q\mid\P).
\]
We first compute the concave biconjugate of $\Grow$ on $\Delta_n$. Write $q_i=\Q(\{i\})$. Then \[ \Grow(\Q) = \sum_{i=1}^n q_i\log\frac{q_i}{p_i}. \]
For $g=(g_1,\ldots,g_n)\in\mathbb R^n$, the concave conjugate of $\Grow$ is \[ a_{\Grow}(g) = \inf_{q\in\Delta_n} \left( \sum_{i=1}^n q_i(g_i+\log p_i) - \sum_{i=1}^n q_i\log q_i \right). \]
Since $-\sum_i q_i\log q_i\ge0$, this expression is bounded below by $\min_i(g_i+\log p_i)$, and this lower bound is attained by choosing $q$ to be a Dirac mass at a minimizer of $g_i+\log p_i$. Hence \[ a_{\Grow}(g)=\min_{1\le i\le n}(g_i+\log p_i). \]
Therefore \begin{align*} \Grow^{**}_{\Delta_n}(\Q) &= \inf_{g\in\mathbb R^n} \left( \sum_{i=1}^n q_i g_i - \min_{1\le i\le n}(g_i+\log p_i) \right) = -\sum_{i=1}^n q_i\log p_i + \inf_{h\in\mathbb R^n} \left( \sum_{i=1}^n q_i h_i-\min_i h_i \right), \end{align*} where $h_i=g_i+\log p_i$. The last infimum is zero, since $\sum_i q_i h_i\ge\min_i h_i$ and equality is attained by taking $h_1=\cdots=h_n$. Thus 
\[ \Grow^{**}_{\Delta_n}(\Q)= -\sum_{i=1}^n \Q(\{i\})\log p_i, \qquad \Q\in\Delta_n. \]

Consequently, 
\[ \Grow(\Q)-\Grow^{**}_{\Delta_n}(\Q) = \sum_{i=1}^n q_i\log q_i. \]
This quantity is minimized over $\Delta_n$ at the uniform distribution $\mathsf U$, and the minimum value is $-\log n$. Thus the offset relative entropy side of \eqref{eq:weak-star-offset-form} has value $-\log n$.

On the e-variable side, let $ X={d\mathsf U}/{d\P}$.  Then $X$ is an e-variable for $\P$, and $X(i)=1/(np_i)$. For $\Q\in\Delta_n$, \begin{align*} \E_\Q[\log X]-\Grow(\Q) &= \sum_{i=1}^n q_i\log\frac{1}{np_i} - \sum_{i=1}^n q_i\log\frac{q_i}{p_i} = -\log n-\sum_{i=1}^n q_i\log q_i. \end{align*} Taking the infimum over $\Q\in\Delta_n$ gives $-\log n$, attained at any vertex of $\Delta_n$. Hence the REGROW value is $-\log n$, and the REGROW e-variable is $d\mathsf U/d\P$. Notice that the minimizing $\Q$ on the e-variable side may be any Dirac mass, whereas the minimizing $\Q$ on the offset relative entropy side is the uniform distribution $\mathsf U$.
\end{example}

\section{Strong duality for unbounded GROW }\label{sec:grow}

We have been unable to derive an ``assumption-free'' theorem for unbounded GROW. The absence of such a theorem forces us to look for special cases where strong duality still holds with unbounded e-variables. While there is no hope of being complete, we present five important settings where strong duality holds. We present these below in the following order, one case per subsection. First, we handle the case of a singleton $\Pcal= \{\P\}$ together with convex $\Qcal$ such that $H(\Q \mid \P) < \infty$ for all $\Q \in\Qcal$. We then extend this to the case where a joint information projection (defined later) exists. Third, we handle the case of convex and setwise compact $\Qcal$. Fourth, we handle the case of weakly compact $\Pcal$ and $\Qcal$ on a Polish space $\Omega$. Finally, we handle the case of finite $\Omega$. In all cases, the final conclusions avoid finitely additive measures. In all relevant cases, we show using counterexamples that the conclusions of the theorems do not hold without restrictions, justifying our inability to find an assumption-free statement.

\subsection{When $\Pcal = \{\P\}$ is a singleton}
Fix a convex family $\Qcal$ of probability measures and a simple null $\Pcal = \{\P\}$ such that $$\inf_{\Q\in \Qcal} H(\Q\mid \P) < \infty.$$  
Recall that then Csisz\'ar's I-projection (IPr), provided it exists, is the probability measure $\Q^{\mathrm{IPr}}\in \Qcal$ that achieves $\inf_{\Q\in\Qcal} H(\Q\mid \P)$ \citep{csiszar1975source}, so that
\[
H(\Q^{\mathrm{IPr}} \mid \P) = \inf_{\Q\in\Qcal} H(\Q\mid \P).
\]
Csisz\'ar proved that the I-projection of $\P$ onto $\Qcal$ exists if $\Qcal$ is convex and TV-closed and $\inf_{\Q\in \Qcal} H(\Q\mid \P) < \infty$, in which case clearly $H(\Q^{\mathrm{IPr}}\mid \P) < \infty$.
Since the IPr does not always exist, the following generalization was conceived.

Following \citet[Theorem~8]{topsoe1979information} and \cite{csiszar1984sanov}, the generalized I-projection (GIPr) $\Q^{\mathrm{GIPr}}$ of $\P\in\Mone$ onto $\Qcal\subseteq\Mone$ is the unique probability measure that satisfies the following Pythagorean inequality for any $\Q \in \Qcal$:
\begin{equation}\label{eq:pythagorean}
 H(\Q\mid \Q^{\mathrm{GIPr}}) + 
\inf_{\RR\in\Qcal}H(\RR\mid \P) \leq H(\Q\mid \P).
\end{equation}
The GIPr $\Q^{\mathrm{GIPr}}$ was shown by the above authors to always exist for any convex $\Qcal$ with $\inf_{\Q\in \Qcal} H(\Q\mid \P) < \infty$. However, 
 $\Q^{\mathrm{GIPr}}\notin \Qcal$ in general, but \eqref{eq:pythagorean} implies that it does lie in the I-closure of $\Qcal$, given by
 \[
    \overline\Qcal^I = \left\{\RR \in \Mone: \inf_{\Q \in \Qcal} H(\Q \mid \RR)=0\right\}. 
\]
Further, one has
\begin{equation}\label{eq:existence-LR}
H(\Q^{\mathrm{GIPr}}\mid \P)\le \inf_{\Q\in\Qcal}H(\Q\mid \P), 
\end{equation}
where strict inequality can occur as shown by \citep[Example~3.2]{csiszar1984sanov}. He also shows that one cannot obtain the generalized I-projection as a minimizer over some closure of $\Qcal$. 
Thanks to \eqref{eq:existence-LR}, and our standing assumption that $\inf_{\Q\in \Qcal} H(\Q\mid \P) < \infty$, we have $H(\Q^{\mathrm{GIPr}}\mid \P) < \infty$, and thus $\Q^\mathrm{GIPr} \ll \P$, meaning that $d\Q^\mathrm{GIPr}/d\P$ is well defined, and it is clearly an e-variable for $\P$. We now show that under some conditions,  $d\Q^\mathrm{GIPr}/d\P$ is the log-optimal (unbounded) e-variable.

\begin{assumption}\label{ass:singleton-P}
    $\inf_{\RR\in \Qcal} H(\RR\mid \P) < \infty$ and there exists a random variable $X^*$ that is a $\P$-version of $d\Q^\mathrm{GIPr}/d\P$, 
    such that
$\E_\Q[\log X^*] \geq \inf_{\RR\in\Qcal}H(\RR\mid \P)$  for each $\Q \in \Qcal$.
\end{assumption}

The above assumption can sometimes be checked by verifying the stronger condition that for every $\Q \in \Qcal$, \(
H(\Q \mid \Q^{\mathrm{GIPr}}) < \infty
\) holds, which in turn is implied by $H(\Q \mid \P) < \infty$. (Indeed, one can rearrange the Pythagorean inequality~\eqref{eq:pythagorean} even when $H(\Q \mid \P)=\infty$, and note that these stronger conditions imply $\Q \ll \Q^{\mathrm{GIPr}} \ll \P$.)

\begin{theorem}[Singleton null strong duality]\label{thm:singleton-P}
For singleton $\Pcal = \{\P\}$ and convex $\Qcal$ such that Assumption~\ref{ass:singleton-P} holds, the corresponding e-variable $X^*$ achieves the supremum on the left-hand side of the following strong duality:
\begin{equation*}
\Grow
= \inf_{\Q \in \Qcal} \E_\Q[\log X^*] = \inf_{\Q\in\Qcal}H(\Q\mid \P).
\end{equation*}
Further, if the GIPr is an IPr (i.e., it lies in $\Qcal$), then $X^*$ is the $\P$-a.s.\ unique GROW e-variable.
\end{theorem}
\begin{proof}
We have
\[
\inf_{\Q\in\Qcal}H(\Q\mid \P) \leq 
\inf_{\Q \in \Qcal} \E_\Q[\log X^*] \leq 
\sup_{X \in \Ecal} \inf_{\Q \in \Qcal} \E_\Q[\log X]  \leq  \inf_{\Q\in\Qcal}H(\Q\mid \P)< \infty ,
\]
where the first inequality follows from Assumption~\ref{ass:singleton-P},  the second one holds by noting $X^* \in \Ecal$, the third by weak duality~\eqref{weak_duality_general}, and the last by assumption. This yields the strong duality. The uniqueness claim follows by Proposition~\ref{prop:version-of-LR}. 
\end{proof}

An interesting application of this theorem, where $\Qcal$ is defined by a finite number of moment inequalities, is presented in~Section~\ref{SS:moments}.

The above result also corrects some inaccuracies in the statement and proof of Proposition~1 in~\cite{grunwald2024growth}. In particular, the inequality in their Equation~(1.11) would only follow if the GIPr belonged to $\Qcal$, which need not hold in general. This leads them to conclude in their Equation~(1.10) that $\Grow$ equals $H(\Q^{\mathrm{GIPr}} \mid \P)$, which is not true in general. It is worth noting that their assumption is \emph{weaker} than the usual Pythagorean inequality, because their Equation~(1.6) is equivalent to assuming our~\eqref{eq:pythagorean} but with the term $\inf_{\RR \in \Qcal} H(\RR \mid \P)$ being replaced by the quantity $H(\Q^{\mathrm{GIPr}} \mid \P)$, which~\eqref{eq:existence-LR} implies can be smaller. Thus the error leads to a weaker assumption but a stronger conclusion. Example~\ref{eg:csiszar} provides an explicit counterexample: it satisfies their assumptions but not their conclusion.

It is natural to ask whether we can relax the assumption to only needing $\inf_{\Q\in\Qcal}H(\Q\mid \P) < \infty$ (which was anyway required for the GIPr to exist). The following example shows that this condition does not suffice, and we can still have
\begin{equation}\label{eq:strict-inequality}
\sup_{X \in \Ecal} \inf_{\Q \in \Qcal} \E_\Q[\log X] <  \inf_{\Q\in\Qcal}H(\Q\mid \P).
\end{equation}

\begin{example}\label{ex:singletonP}
Let $\Omega=[0,1]$, equipped with the Borel sigma algebra, and consider $\P=\mathsf U$, the uniform. Define $\Q_0$ by
$
{d\Q_0}/{d\P}(\omega)=2\omega
$
so that $\Q_0$ is a probability measure, and set
\[
\Qcal=\conv(\{\Q_0\}\cup\{\delta_\omega:\omega\in[0,1]\}).
\]
Let $X$ be any e-variable with respect to $\P$. Then necessarily $\inf_{\omega\in[0,1]}X(\omega)\le1$; otherwise
$X>1$ everywhere and hence $\E_\P[X]>1$. Therefore, the
left-hand side of~\eqref{eq:strict-inequality} equals zero, achieved by the constant e-variable $1$. 
 On
the other hand, among the elements of $\Qcal$ the only one absolutely
continuous with respect to $\P$ is $\Q_0$, while
$
H(\Q_0\mid\P)
=
\log 2-1/2
>0$.
Hence the right-hand side of \eqref{eq:strict-inequality} is finite and
strictly positive, being attained by $\Q_0$.
\end{example}

The above example does not have $\Q \ll \P$ for every $\Q \in \Qcal$. So it leaves open the possibility that assuming $\inf_{\Q\in\Qcal}H(\Q\mid \P) < \infty$ along with $\Q \ll \P$ for every $\Q \in \Qcal$ may possibly suffice. However, such hopes are quickly dashed by an extension of the above example, which due to its length is provided later as Example~\ref{ex:singletonP-continued}.

We end by noting that Assumption~\ref{ass:singleton-P} is not necessary for strong duality 
to hold. Rather, it additionally ensures that the supremum over e-variables is attained. The following example  emphasizes this point. (See also Example~\ref{ex:260621} for a more statistically relevant instance.)

\begin{example}\label{eg:jipr-likelihoodratio-fail}
    Let $\Omega = \N \cup \{0\}$, equipped with its power set sigma algebra, and let $\Pcal = \{\P\}$ with $\P(n) = 2^{-(n+1)}$ for all $n \in \N_0$. Define two probability measures $\Q_0 = \delta_0$ and $\Q_\infty(n) = 1/(n(n+1))$ for $n \in \N$, and set 
    \[
        \Qcal = \conv(\Q_0,\Q_\infty) = \{t\Q_0+(1-t)\Q_\infty:t\in[0,1]\}.
    \]
    We have $H(\Q_0 \mid \P) = \log 2$, but $H(\Q_\infty \mid \P) = \infty$, so for any $t \in [0,1)$, 
    \[H(t \Q_0 + (1-t)\Q_\infty \mid \P) = \infty.\]
    Consequently, $\Q_0$ is the IPr of $\P$, with relative entropy $\log 2$. The likelihood ratio $X^* = d\Q_0/d\P$ is uniquely determined since $\P$ has full support. It satisfies $X^*(0) = 2$ and $X^*(n) = 0$ for all $n \in \N$. Thus, $\E_{\Q_\infty}[\log X^*] = -\infty$, in particular $\inf_{\Q \in \Qcal} \E_\Q[\log X^*] = -\infty$. To see that strong duality holds, for a small $\varepsilon > 0$, define the e-variable $X^\varepsilon$ by $X^{\varepsilon}(0) = 2(1-\varepsilon)$ and $X^{\varepsilon}(n) = \varepsilon \Q_\infty(n)/\P(n)$. Further, $\E_{\Q_\infty}[\log X^\varepsilon]=\infty$, but $\E_{\Q_0}[\log X^\varepsilon]=\log(2(1-\varepsilon))$. Thus, $\inf_{\Q \in \Qcal} \E_\Q[\log X^\varepsilon] = \log(2(1-\varepsilon))$ which approaches $\log 2$ as $\varepsilon \downarrow 0$, meaning that strong duality holds. (In fact, Proposition~\ref{prop:version-of-LR}, proved later, implies that no e-variable achieves strong duality.)
\end{example}

\begin{remark}\label{rem:mixture-epower}
We explain the key idea behind the above example. Let $\Q^*$ denote the generalized I-projection of $\P$ onto some convex $\Qcal$ with $I = \inf_{\Q \in \Qcal} H(\Q \mid \P) < \infty$. Let  $X^*$ be any fixed version of $d\Q^*/d\P$; clearly $X^* \in \Ecal$. Call a $\Q \in \Qcal$ ``bad'' if $\E_\Q[\log X^*] < I$, and ``good'' otherwise. Suppose there exists an e-variable $Y$ which  satisfies $\E_\Q[\log Y] = \infty$ for every bad $\Q \in \Qcal$. Then for $X_\delta = (1-\delta)X^* + \delta Y \in \Ecal$, the mixture with $Y$ provides protection against the bad $\Q$. Since $X_\delta \geq \delta Y$, we have $\E_\Q[\log X_\delta] \geq \E_\Q[\log Y] + \log \delta = \infty$ for bad $\Q$. And for the ``good'' $\Q$, $X_\delta \geq (1-\delta)X^*$ gives us $\E_\Q[\log X_\delta] \geq I + \log(1-\delta)$. Put together, we get that $\inf_{\Q \in \Qcal} \E_{\Q} [\log X_\delta] \geq I + \log(1-\delta)$, and letting $\delta \downarrow 0$ yields strong duality.  We formalize these ideas further in Theorem~\ref{T:shielding}.
\end{remark}

\subsection{When a joint information projection exists}\label{subsec:jipr}

\cite{larsson2025numeraire} showed that for an arbitrary composite null $\Pcal$ and singleton alternative $\Q$, a  log-optimal e-variable $X^\mathrm{num}$, called the numeraire, always exists and is $\Q$-almost-surely unique and positive. Then, they define the reverse information projection (RIPr) of $\Q$ onto $\Peff$ as the sub-probability measure defined by $d\P^*/d\Q = 1/X^\mathrm{num}$, and show that $\P^* \in \Peff$.

We define the GIPr (or IPr, if it exists) of a nonzero sub-probability $\P \in \Mplus$  onto a convex set $\Qcal$, as the GIPr (or IPr) of $\tilde \P = \P/\P(\Omega)$  onto $\Qcal$ (assuming, as before, that $\inf_{\Q \in \Qcal} H(\Q \mid \tilde \P) < \infty$). Since the Pythagorean inequality~\eqref{eq:pythagorean} contains $\tilde \P$ on both sides, one finds that the normalization constant cancels out. 
Hence the inequality holds for $\P$, despite $\P$ being a sub-probability, a fact that we record as a lemma for easier reference.

\begin{lemma}
    Let $\Qcal$ be convex. For any nonzero sub-probability measure $\P \in \Mplus$ with $\inf_{\Q\in \Qcal} H(\Q\mid \P) < \infty$, its GIPr satisfies~\eqref{eq:pythagorean} verbatim, i.e.,
    \begin{equation*}
H(\Q\mid \Q^{\mathrm{GIPr}}) + 
\inf_{\RR\in\Qcal}H(\RR\mid \P) \leq H(\Q\mid \P).
\end{equation*}
\end{lemma}

Now, we can define the joint information projection (JIPr).

\begin{definition}\label{def:JIPr}
Given arbitrary $\Pcal$ and convex $\Qcal$, we will say that a pair $(\P^*,\Q^*) \in \Mplus \times\Mone$ is a joint information projection (JIPr) if 
\begin{enumerate}
\item[(i)] $\P^* \in \Peff$ is the RIPr of $\Q^*$ onto $\Peff$; 
\item[(ii)] $\inf_{\Q \in \Qcal} H(\Q \mid \P^*) < \infty$ and $\Q^* \in \overline\Qcal^I$ is the GIPr of $\P^*$ onto $\Qcal$.
\end{enumerate}
\end{definition}

Also, assuming (i,ii) is weaker than assuming
\begin{enumerate}
    \item[(iii)] $(\P^*,\Q^*) \in \Peff \times \Qcal$ achieves $\inf_{\Q\in\Qcal}\inf_{\P\in\Peff} H(\Q\mid\P)$, this value is finite, and $\P^* \ll \Q^*$. 
\end{enumerate}
One has the implication (iii) $\implies$ (i,ii), but in general the reverse implication may fail. We provide an example to illustrate that (i,ii) does not imply (iii), even when the GIPr is an IPr.

\begin{example}\label{eg:jipr-definitions}
Let $\Omega=\{0,1,2\}$ and set
\(
\P^*=\frac12\delta_1+\frac12\delta_2\) and $
\Q^*=\delta_1$. 
Define the convex sets
\[
    \Pcal=
\{(1-s)\P^*+s\delta_0:s\in[0,1]\};
\qquad
\Qcal
=
\{(1-s)\Q^*+s\delta_0:s\in[0,1]\}.
\]
Then the absolutely continuous part of $\P^*$ with respect to $\Q^*$ is
\(
\P^*_a=\frac12\delta_1.
\)
We will show that the pair $(\P^{*}_a, \Q^{*})$ satisfies both (i) and (ii).
Since the e-variable $X^* = 2 \1_{\{1\}}$ is the reciprocal of $d\P_a^* / d\Q^*$, 
$\P^*_a$ is the RIPr of $\Q^*$ onto
$\Peff$ \citep[Theorem~4.1]{larsson2025numeraire}, so condition (i) holds.
Next, $\Q^*$ is the IPr of $\P^*_a$ onto $\Qcal$. 
Indeed,
\(
H(\Q^*\mid\P^*_a)=\log 2<\infty,
\)
whereas every element $\Q_s = (1-s)\Q^*+s\delta_0$ with $s>0$ assigns positive mass
to $\{0\}$, so $H(\Q_s \mid \P^*_a) = \infty$. Hence condition~(ii) holds.
However, condition (iii) fails. Indeed, $\delta_0\in\Pcal \cap \Qcal$, and therefore
\(
\inf_{\Q\in\Qcal}\inf_{\P\in\Peff}H(\Q\mid\P)=0.
\)
On the other hand,
\(
H(\Q^*\mid\P^*_a)=\log 2>0.
\)
Thus $(\P^*_a,\Q^*)$ does not attain (iii), even though conditions (i) and (ii) hold.   
\end{example}

Another illustration for (i) and (ii) holding but (iii) failing is given in Example~\ref{eg:csiszar}. This example is slightly longer, but also satisfies Assumption~\ref{ass:composite-P} below, hence Theorem~\ref{T:260608} applies for that example.

Condition~(i) implies that $d\Q^*/d\P^*$ is an e-variable for $\Pcal$ (it is the numeraire for $\Pcal$ versus $\Q^*$). However, it may not achieve the desired strong duality; recall Example~\ref{ex:singletonP}, where $(\P,\Q_0)$ was a JIPr pair satisfying condition (iii). 
Thus, for strong duality to hold, we need an additional condition, analogous to Assumption~\ref{ass:singleton-P}.

\begin{assumption}\label{ass:composite-P}
A JIPr pair $(\P^*, \Q^*)$ exists and there exists a random variable $X^*$ that is a $\P^*$-version  of $d\Q^*/d\P^*$, such that
$\E_\Q[\log X^*] \geq \inf_{\RR\in\Qcal}H(\RR\mid \P^*)$ for each $\Q \in \Qcal$.
\end{assumption}

As before, one can sometimes check the above by verifying that either $H(\Q \mid \Q^*) < \infty$ holds, or the even stronger condition that $H(\Q \mid \P^*) < \infty$ holds, for each $\Q \in \Qcal$. (To see that these imply the assumption, simply rearrange \eqref{eq:pythagorean}, which is allowed even when $H(\Q \mid \P^*)=\infty$).

\begin{theorem}[Strong duality for composite null and alternative]\label{T:260608}
For arbitrary $\Pcal$ and convex $\Qcal$,
suppose there exists a joint information projection $(\P^*,\Q^*)$ such that Assumption~\ref{ass:composite-P} holds. Then the corresponding e-variable $X^*$ is a GROW e-variable, which satisfies the strong duality
    \begin{align*} \label{eq:260619}
    \Grow = \sup_{X \in \Ecal}\inf_{\Q \in \Qcal} \E_{\Q}[\log X] = \inf_{\Q \in \Qcal} \E_\Q [\log X^*] = \inf_{\Q \in \Qcal} H(\Q \mid \P^*) =  \inf_{\Q \in \Qcal} \min_{\P \in \Peff} H(\Q \mid \P).
    \end{align*}
Further, if $\Q^*$ is an IPr (i.e., it lies in $\Qcal$), then $X^*$ is the $\P^*$-a.s.\ unique GROW e-variable.
\end{theorem}

The above theorem eliminates certain conditions in~\citet[Theorem~1, First Generalisation]{grunwald2024safe}, such as assuming that $H(\Q \mid \Q') < \infty$ for all $\Q,\Q' \in \Qcal$, and also that these have full support.

We note above that if the GIPr actually was an IPr (meaning that the infimum relative entropy was achieved by $\Q^* \in \Qcal$), then we can add to the strong duality a further equality to $H(\Q^* \mid \P^*)$.

\begin{proof}[Proof of Theorem~\ref{T:260608}]
We begin with the weak duality
\[\sup_{X\in\Ecal}
\inf_{\Q\in\Qcal}\E_\Q[\log X]
\leq \inf_{\Q\in\Qcal} \sup_{X\in\Ecal}
\E_\Q[\log X] = \inf_{\Q\in\Qcal} \min_{\P \in \Peff} H(\Q\mid\P) \leq \inf_{\Q\in\Qcal} H(\Q \mid \P^*),
\]
where the sole equality follows from~\eqref{eq:numeraire-duality}, and the two inequalities are immediate.

In the opposite direction, we have
\[
\sup_{X \in \Ecal }\inf_{\Q\in\Qcal}\E_\Q[\log X]
\geq
\inf_{\Q\in\Qcal}\E_\Q[\log X^*]
\geq
\inf_{\Q \in \Qcal} H(\Q \mid \P^*),
\]
where the first inequality is immediate, and the second follows by taking infimum over $\Qcal$ in Assumption~\ref{ass:composite-P}.
Combining the two displays completes the proof of strong duality. The uniqueness claim follows by Proposition~\ref{prop:version-of-LR}. 
\end{proof}

\begin{example}
    Let $\Pcal$ be $\{N(m,1):m\leq0\}$ and $\Qcal$ be (the convex hull of) $\{N(m,1):m\geq1\}$. Then the JIPr pair is given by $(N(0,1),N(1,1))$ and the GIPr is actually an IPr. In this case, the JIPr pair is also a least favorable distribution pair in the sense of~\cite{Huber:Strassen:1973}; indeed \cite{saha2025huber} already applies to yield the GROW e-variable, which is the likelihood ratio of the JIPr pair.
    
    Now take
    $\Pcal$ to be set of all $1$-sub-Gaussian distributions with nonpositive mean, and $\Qcal$ to be (the convex hull of) all $1$-sub-Gaussian distributions with mean at least $1$ and finite relative entropy to the standard Gaussian.
    For this nonparametric class, a least favorable distribution pair in the sense of Huber--Strassen  does not exist, and it is not true that $H(\Q\mid \Q') < \infty$ for all $\Q,\Q'\in \Qcal$ (take, for example, $\Q=\mathsf{U}[1,2],\Q'=\mathsf{U}[3,4]$), so the results in~\cite{grunwald2024safe} do not apply. Nevertheless, the JIPr pair is still $(N(0,1),N(1,1))$, and the conditions of Theorem~\ref{T:260608} are fulfilled, so the GROW e-variable is still their likelihood ratio.
\end{example}

\begin{proposition}\label{prop:version-of-LR}
    Assume that the following strong duality holds:
        \begin{align*} 
    \Grow = \sup_{X \in \Ecal}\inf_{\Q \in \Qcal} \E_{\Q}[\log X] =  \inf_{\Q \in \Qcal} \inf_{\P \in \Peff} H(\Q \mid \P) < \infty,
    \end{align*} 
    and that there exists    
 $(\P^*,\Q^*) \in \Peff \times \Qcal$  such that $H(\Q^* \mid \P^*) = \inf_{\Q \in \Qcal} \inf_{\P \in \Peff} H(\Q \mid \P)$.
    If any e-variable $X^*$ achieves the supremum above, then we must have $X^*={d\Q^*}/{d\P^*}$,  $\P^*\text{-a.s.}$
\end{proposition}
\begin{proof} 
Since $\E_{\Q^*}[\log X^*]\ge \Grow$ because $X^* \in \Ecal$ witnesses strong duality, applying the Donsker--Varadhan inequality  (Lemma~\ref{lem:extended-entropy-sec}) yields
    \[
      \Grow \leq \E_{\Q^*}[\log X^*] \leq H(\Q^* \mid \P^*) + \log \E_{\P^*}[X^*] \leq H(\Q^* \mid \P^*) = \Grow.
    \] 
It follows that equality holds throughout. In particular, $\log\E_{\P^*}[X^*]=0$, hence $\E_{\P^*}[X^*]=1$, and 
$
\E_{\Q^*}[\log X^*]
=
H(\Q^*\mid\P^*)$. 

Define now a probability measure $\Q'$ by
$
\frac{d\Q'}{d\P^*}=X^*$.
Since $H(\Q^*\mid\P^*)<\infty$, we have $\Q^*\ll\P^*$. Therefore, using the
chain rule for relative entropy,
\[
H(\Q^*\mid\Q')
=
H(\Q^*\mid\P^*)-\E_{\Q^*}[\log X^*]
=
0.
\]
Thus $\Q^*=\Q'$, proving the claim.
\end{proof}

\begin{remark}
    If $\Q^*$ in Proposition~\ref{prop:version-of-LR} is not in $\Qcal$ but in  the I-closure $\overline \Qcal^I$ only, then the conclusion of the statement still holds provided one assumes in addition $\E_{\Q^*}[\log X^*]\ge \Grow$.
\end{remark}

We recall Example~\ref{eg:jipr-likelihoodratio-fail}, which points out that Assumption~\ref{ass:composite-P} is not necessary. Indeed, this example shows that the JIPr can exist (satisfying the stronger condition (iii)) --- it uniquely equals $(\delta_0/2, \Q_0)$ --- and strong duality can hold, but the likelihood ratio of the JIPr pair does not achieve the strong duality.

\subsection{When $\Qcal$ is setwise compact}\label{subsec:grow}

\begin{definition}\label{def:setwise-topology}
The \emph{setwise topology} on $\mathcal M$ is the initial topology $\sigma(\mathcal M,\Bb)$ induced by the maps 
\[ \RR\mapsto \int f\,d\RR, \qquad \RR\in\Mcal, f\in\Bb. 
\]
A set $K\subseteq\mathcal M$ is called setwise closed/compact if it is closed/compact in  $(\mathcal M,\sigma(\mathcal M,\Bb))$.  
\end{definition}

The setwise topology is the subspace topology induced by the weak-$*$ topology in $\ba$: \[ \sigma(\mathcal M,\Bb) = \sigma(\ba,\Bb)\big|_{\mathcal M}. \] Indeed, $\sigma(\ba,\Bb)$ is the coarsest topology on $\ba$ making all maps \( \nu\mapsto \int f\,d\nu, f\in\Bb, \) continuous. Restricting these maps to $\mathcal M\subseteq\ba$ gives exactly the maps defining $\sigma(\mathcal M,\Bb)$.

If we equip $\mathcal M$ with the total variation (TV) norm $\|\cdot\|_{\mathrm{TV}}$ and its norm topology, $\mathrm{TV}$ convergence  implies setwise convergence, because  $\|\RR_n - \RR\|_{\mathrm{TV}}\to 0$ implies $\RR_n\to\RR$ setwise. In addition, if we further suppose that $\Omega$ is Polish and $\mathcal F$ is the Borel sigma algebra, then setwise convergence implies weak convergence: clearly, $\RR_n \to \RR$ in $\sconv$ implies that $\RR_n\to \RR$ in $\sigma(\mathcal M, C_b)$.
All three topologies are Hausdorff.

Recall that $\Qcal\subseteq\Mone$ is said to be uniformly absolutely continuous  with respect to a finite $\RR \in \Mplus$ if $\Q\ll \RR$ for all $\Q\in\Qcal$, and for every $\varepsilon>0$ there exists $\delta>0$ such that for all $A\in\mathcal F$,
$
\RR(A)<\delta$ implies $
\sup_{\Q\in\Qcal} \Q(A)\le \varepsilon$.
It follows directly from \cite[Theorem~4.7.25]{Bogachev2007Measure}  that 
    $\cl^{\sconv}(\co(\Qcal))$ is setwise compact if and only if
    there exists a finite measure $\RR\in\Mplus$ such that $\Qcal$ is uniformly absolutely continuous with respect to $\RR$.
    From this we see that assuming $\Qcal$ to be convex and setwise compact is a rather strong assumption.

Two examples of convex and setwise compact $\Qcal$ are (i) the convex hull of Gaussians with means in $\{-1,1\}$ and unit variance and (ii) all Bernoulli distributions with parameters in $[0.2,0.8]$.

We shall use the following two lemmata.
\begin{lemma}\label{lem:weak-star-setwise}
If $\Qcal$ is convex and setwise compact, then $\cconv^*(\Qcal) = \Qcal$. 
\end{lemma}
\begin{proof}
    Since the setwise
topology on $\Mone$ is the subspace topology induced by
$\sigma(\ba,\Bb)$, the set $\Qcal$ is compact also as a subset of
$(\ba,\sigma(\ba,\Bb))$. The latter space is Hausdorff, so compact subsets are
closed. Hence $\Qcal$ is $\sigma(\ba,\Bb)$-closed, yielding the statement.
\end{proof}

\begin{lemma}[Distribution-uniform monotone convergence]\label{lem:uniform-mct}
If $\Qcal$ is setwise compact then
\begin{equation}\label{eq:dumc}
\text{$f_n \ge 0$ and $f_n \uparrow f$} \quad \Longrightarrow \quad \inf_{\Q \in \Qcal} \E_\Q[f_n] \uparrow \inf_{\Q \in \Qcal} \E_\Q[f].
\end{equation}
Further, if $\eqref{eq:dumc}$ holds, then $\Qcal$ is relatively setwise compact and $\cconv^*(\Qcal) \subseteq \Mone$.
\end{lemma}
\begin{proof}
Set
\[
a_n=\inf_{\Q\in\Qcal}\E_\Q[f_n],
\qquad
a=\inf_{\Q\in\Qcal}\E_\Q[f].
\]
Then $a_n$ is increasing and $a_n\le a$. Let
$L=\lim_n a_n$. 
Suppose, to the contrary, that
$L<a$.  
For each $n$, choose $\Q_n\in\Qcal$ such that
\[
\E_{\Q_n}[f_n]\le a_n+\frac1n.
\]
By setwise compactness, there is a subnet $\Q_{n_\alpha}$ converging
setwise to some $\Q^*\in\Qcal$. Since $n_\alpha$ is a subnet of the
sequence of indices, we have $n_\alpha\uparrow\infty$.

Fix now $k,m\in\N$. For all sufficiently large $\alpha$, $n_\alpha\ge k$, and hence
\[
\E_{\Q_{n_\alpha}}[f_k\wedge m]
\le
\E_{\Q_{n_\alpha}}[f_{n_\alpha}]
\le
a_{n_\alpha}+\frac1{n_\alpha}.
\]
Since $f_k\wedge m\in\Bb$, setwise convergence yields
$
\E_{\Q^*}[f_k\wedge m]\le L$.
By monotone convergence we get
$
\E_{\Q^*}[f_k]\le L$.
Letting now $k\uparrow\infty$ and using monotone convergence once more gives
$
\E_{\Q^*}[f]\le L$.
Therefore
\[
a=\inf_{\Q\in\Qcal}\E_\Q[f]\le \E_{\Q^*}[f]\le L,
\]
contradicting $L<a$. Hence $L=a$, proving the claim.

Let $A_n\downarrow\varnothing$. Applying \eqref{eq:dumc} to
\(
    f_n=\1_{A_n^c}\uparrow1
\)
gives
\[
\begin{aligned}
    1-\sup_{\Q\in\Qcal}\Q(A_n)
    &=
    \inf_{\Q\in\Qcal}\Q(A_n^c) \uparrow
    \inf_{\Q\in\Qcal}\Q(\Omega)
    =
    1.
\end{aligned}
\]
Hence
$
    \sup_{\Q\in\Qcal}\Q(A_n)\downarrow 0$.
Now let $\mu \in \cconv^*(\Qcal)$.
Thus there is a net $\Q_\alpha$ in $\co(\Qcal)$ that weak-$*$ converges to $\mu$. For each fixed $n$, since
$\1_{A_n}\in \Bb$,
\[
    \mu(A_n)
    =
    \lim_\alpha \Q_\alpha(A_n)
    \le
    \sup_{\Q\in\Qcal}\Q(A_n) \downarrow 0
    \qquad\text{whenever }A_n\downarrow\varnothing.
\]
Thus $\mu$ is continuous from above, hence 
$\mu\in \Mone$, yielding $\cconv^*(\Qcal) \subseteq \Mone$.
 Since $\ba_1$ is weak-$*$ compact, this closure is weak-$*$
compact. The weak-$*$ topology restricted to $\Mone$ is exactly the
setwise topology, and hence the setwise closure of $\Qcal$ is
compact. Thus $\Qcal$ is relatively setwise compact.
\end{proof}

With these two lemmata in place we  can now prove the following.

\begin{theorem}\label{thm:setwise}
If $\Pcal$ is arbitrary and $\Qcal$ is convex and setwise compact, then
\[
\Grow= \Grow_b =\min_{\Q\in \Qcal}\min_{\P\in\Peff} H(\Q\mid \P).
\]
\end{theorem}

The proof below clarifies that distribution-uniform monotone convergence \eqref{eq:dumc} is the only property of $\Qcal$ that we need in order for $\Grow=\Grow_b=\min_{\Q\in \cconv^*(\Qcal) \cap \Mone}\min_{\P\in\Peff} H(\Q\mid \P)$ to hold. 
(Indeed, Lemma~\ref{lem:uniform-mct} shows that under~\eqref{eq:dumc}, the weak-$*$ closure of $\Qcal$ lies in $\Mone$, and then Corollary~\ref{C:260526} allows to replace $\cconv^*(\Pcal)$ with $\Peff$.) 
Setwise compactness only helps to further restrict to just $\Qcal$ on the right hand side.

    \begin{proof}[Proof of Theorem~\ref{thm:setwise}]
    We first argue that 
    \(
    \Grow_b = \min_{\Q\in  \Qcal}\min_{\P\in\Peff} H(\Q\mid \P).
    \)
Lemma~\ref{lem:weak-star-setwise} allows us to replace $\cconv^*(\Qcal)$ in Theorem~\ref{T_strong_duality} with just $\Qcal$. Corollary~\ref{C:260526} shows that for $\Q \in \Mone$,
$\min_{\mu \in \cconv^*(\Pcal)} H(\Q\mid\mu)= \min_{\P\in\Peff} H(\Q\mid \P)$.
This proves $\Grow_b=\min_{\Q\in  \Qcal}\min_{\P\in\Peff} H(\Q\mid \P)$.

    Clearly \(\Grow_b\le \Grow\), so we focus on the reverse inequality. Fix an arbitrary e-variable \(X\). For \(\varepsilon\in(0,1)\) and
\(n\in\mathbb N\), define
\[
X_{\varepsilon,n}
=
\varepsilon+(1-\varepsilon)(X\wedge n), \qquad X_\varepsilon = \varepsilon + (1-\varepsilon) X.
\]
Then \(X_{\varepsilon,n}\) is strictly positive and \(X_{\varepsilon,n} \in \Ecal_b\). For fixed \(\varepsilon\), we have
\(
\log X_{\varepsilon,n}\uparrow \log X_\varepsilon
\)
and \(\log X_{\varepsilon,n}\ge \log \varepsilon\). Hence Lemma~\ref{lem:uniform-mct} (applied to the nonnegative functions $\log X_{\varepsilon,n} - \log \varepsilon$) yields
\[
\Grow_b
\ge
\sup_n \inf_{\Q\in\Qcal} \E_\Q[\log X_{\varepsilon,n}]
=
\inf_{\Q\in\Qcal} \E_\Q[\log X_\varepsilon]
\ge
\inf_{\Q\in\Qcal} \E_\Q[\log X]
+
\log(1-\varepsilon).
\]
Letting \(\varepsilon\downarrow 0\), we get
\(
\Grow_b
\ge
\inf_{\Q\in\Qcal} \E_\Q[\log X].
\)
Since \(X \in \Ecal\) was arbitrary, we conclude that
\(
\Grow_b\ge \Grow
\).
\end{proof}

We end this section with a slightly strengthened form of~\citet[Proposition~4.5]{ram2026optimal}. In preparation, define 
the solid hull of $\Pcal$ as
\[
\sol(\Pcal)=\{\S\in \Mplus : \S\le \P \text{ for some } \P\in\Pcal\} \subseteq \Peff.
\]

\begin{proposition}\label{prop:klinf-infkl}
Assume $(\Omega, \mathcal F)$ is a Polish space with its Borel sigma algebra. If $\Pcal$ is convex and weakly compact, then $\Peff$ equals the solid hull of $\Pcal$, and thus for any $\Q \in \Mone$,
    \[
    \min_{\S \in \Peff} H(\Q \mid \S)  = \min_{\P \in \Pcal} H(\Q \mid \P).
    \]
Thus, in Theorems~\ref{T:260608} and \ref{thm:setwise}, if $\Pcal$ is convex and weakly compact, we can replace $\Peff$ by $\Pcal$ on the right-hand side.
\end{proposition}
\begin{proof}
To begin with, on a Polish space, the weak topology on probability measures is metrizable and thus the topology on uniformly bounded subprobabilities is metrizable also. Thus, we may work with sequences instead of nets. 
Weak compactness of $\Pcal$ implies that $\sol(\Pcal)$ is weakly closed. Indeed, suppose $\S_n \in \sol(\Pcal)$ with $\S_n \le \P_n$ for some $\P_n \in \Pcal$ and suppose $\S_n \to \S$. By weak compactness of $\Pcal$, we may pass to a subsequence such that $\P_n \to \P \in \Pcal$. Then, $\P_n - \S_n \to \P - \S$ and since each $\P_n-\S_n$ is positive, so is $\P-\S$. This implies that $\S\le \P$ and thus $\S\in \sol(\Pcal)$. This shows that $\sol(\Pcal)$ is weakly closed. 

We next prove $\sol(\Pcal)\supseteq \Peff$. 
To this end, suppose there exists $\RR \in \Peff$ with $\RR\notin \sol(\Pcal)$. 
Now, $\sol(\Pcal)$ being weakly closed and convex, there exists an $f\in C_b$ such that
\[
\int f^+ d\RR \geq \int f d\RR> \sup_{\S \in \sol(\Pcal)} \int f d\S.
\]
Solidity yields
\[
\sup_{\S\in \sol(\Pcal)} \int f d\S = \sup_{\P\in\Pcal}\int f^+ d\P.
\]
Indeed, this is true since for each $\P\in\Pcal$ the supremum of $\int f d \S$ over $\S \le \P$ is achieved by $\S = \1_{f>0}\P$. Hence, we have the existence of $f \in C_b$ such that
\[
\int f^+ d\RR > \sup_{\P\in \Pcal} \int f^+ d\P.
\]
Now set $M=\sup_{\P\in\Pcal} \int f^+ d\P$. If $M>0$ then $X = f^+ / M \in \Ecal_b$ with $\E_{\RR}[X]>1$. If $M=0$, then $c = \int f^+ d\RR>0$ but $X= 2 f^+ / c \in \Ecal_b$ with $\E_{\RR}[X]=2>1$. This proves that $\RR\notin \Peff$. We conclude that $\Peff = \sol(\Pcal)$. 

Now, if $\S\le \P$, then $H(\Q\mid \S)\ge H(\Q\mid \P)$. Moreover, since $\Pcal\subseteq \Peff$, it follows that $\inf_{\S\in \Peff} H(\Q\mid \S) = \inf_{\P\in\Pcal} H(\Q\mid \P)$. Finally, since the map $\P\mapsto H(\Q\mid \P)$ is weakly lower semicontinuous and $\Pcal$ is weakly compact, we have the last infimum is attained. 
\end{proof}

\subsection{When $\Pcal$ and $\Qcal$  are  convex and weakly compact}\label{subsec:grow-weakly-compact}

In this section, we will prove a countably additive version of our main result by imposing a Polish sample space $\Omega$ and a weak compactness assumption on $\Pcal$ and $\Qcal$, which allows one to work with continuous e-variables and avoid assuming a common dominating reference measure.

To this end, in this section alone we will assume that $\Omega$ is a Polish space, equipped with the Borel sigma algebra. We will equip $\Mone$ with the weak topology $\sigma(\mathcal M, C_b)$ meaning that a sequence $\P_n$ converges weakly to $\P$ if and only if for every $h\in C_b$, $\int h d\P_n \to \int h d\P$. Also, we will write 
$$\Ecal_{bc} =\Ecal_{bb}\cap C_b,$$
where $\Ecal_{bb}$ is the subset of e-variables that are bounded from above and away from zero. We denote the corresponding GROW values as $\Grow_{bb}$ and $\Grow_{bc}$.

\begin{theorem}\label{thm:nondominated-compact-convex}
Suppose that $\Omega$ is Polish, equipped with the Borel sigma algebra, and assume $\Pcal$ and $\Qcal$ are both convex and weakly compact. Then, 
\begin{align*}
\Grow = \Grow_b = \Grow_{bb} = \Grow_{bc} 
=\min_{\Q\in\Qcal}\min_{\P\in\Pcal} H(\Q\mid \P).
\end{align*}
\end{theorem}

By Theorem~\ref{T_strong_duality}, $\Grow_b$ (and hence $\Grow,\Grow_{bc},\Grow_{bb}$) also equals $\min_{\nu\in\cconv^*(\Qcal)}\min_{\mu\in\cconv^*(\Pcal)} H(\nu\mid \mu)$.
The proof uses the Donsker--Varadhan variational representation on Polish
spaces \citep[Lemma 6.2.13]{dembo2009large},
\begin{equation}\label{eq:Donsker-Varadhan-Cb}
H(\Q\mid\P)
=
\sup_{g\in C_b}
\left(
\E_\Q[g]-\log \E_\P[e^g]
\right).
\end{equation}
We first record a robust version over compact convex null classes, obtained by
an application of Sion's minimax theorem.

\begin{lemma}[Robust Donsker--Varadhan formula]
\label{lem:Donsker-Varadhan-convex-compact}
Suppose that $\Omega$ is Polish, equipped with its Borel sigma algebra, and
let $\Pcal$ be convex and weakly compact. Then, for
every $\Q\in\Mone$,
\[
\sup_{g\in C_b}
\inf_{\P\in\Pcal}
\left(
\E_\Q[g]-\log \E_\P[e^g]
\right)
=
\min_{\P\in\Pcal} H(\Q\mid\P).
\]
\end{lemma}

\begin{proof}
Fix $\Q\in\Mone$ and define, for $g\in C_b$ and $\P\in\Pcal$,
$
\Phi(g,\P)
=
\E_\Q[g]-\log \E_\P[e^g]$.
For fixed $g\in C_b$, the map $\P\mapsto \Phi(g,\P)$ is weakly continuous,
because $e^g\in C_b$. It is also convex in $\P$, since
$\P\mapsto \E_\P[e^g]$ is affine and $x\mapsto-\log x$ is convex on
$(0,\infty)$. For fixed $\P\in\Pcal$, the map $g\mapsto \Phi(g,\P)$ is
continuous (with respect to the supremum norm) and concave on $C_b$. Hence Sion's minimax theorem and \eqref{eq:Donsker-Varadhan-Cb} yield
\[
\sup_{g\in C_b}\inf_{\P\in\Pcal}\Phi(g,\P)
=
\inf_{\P\in\Pcal}\sup_{g\in C_b}\Phi(g,\P)
=
\inf_{\P\in\Pcal}H(\Q\mid\P).
\]
Finally, the map $\P\mapsto H(\Q\mid\P)$ is weakly lower semicontinuous, again
by \eqref{eq:Donsker-Varadhan-Cb}, since it is the supremum over $g\in C_b$ of
weakly continuous functions of $\P$. As $\Pcal$ is weakly compact, the infimum
is attained. 
\end{proof}

\begin{proof}[Proof of Theorem~\ref{thm:nondominated-compact-convex}]
The inclusions $\Ecal_{bc}\subseteq\Ecal_{bb}\subseteq\Ecal_b \subseteq \Ecal$, together with \eqref{weak_duality_general}, imply \[ \Grow_{bc}\le \Grow_{bb}\le \Grow_b\le \Grow \leq \inf_{\Q\in\Qcal} \inf_{\P\in\Pcal} H(\Q\mid\P).
\]
It is therefore enough to prove the reverse inequality \[ \Grow_{bc} \ge \inf_{\Q\in\Qcal} \inf_{\P\in\Pcal} H(\Q\mid\P) . \] 
Indeed, this will imply equality throughout. The attainment of the two infima will follow from weak compactness and weak lower semicontinuity of relative entropy.

To make headway, for $g\in C_b$, define
\[
c(g)
=
\log \sup_{\P\in\Pcal}\E_\P[e^g],
\qquad
\Psi(\Q,g)
=
\E_\Q[g]-c(g),
\qquad \Q\in\Qcal.
\]
The function $c$ is finite-valued and continuous. It is also convex since for each
$\P\in\Pcal$ the map $g\mapsto \log \E_\P[e^g]$ is convex and $c$ is the
pointwise supremum of these convex functions.
Consequently, for each $\Q\in\Qcal$, the map
$g\mapsto\Psi(\Q,g)$ is concave and continuous on $C_b$. For each $g\in C_b$, the map
$\Q\mapsto\Psi(\Q,g)$ is affine and weakly continuous on $\Qcal$.

We now claim that
\begin{equation}\label{eq:continuous-payoff-reduction}
\Grow_{bc}
=
\sup_{g\in C_b}
\inf_{\Q\in\Qcal}
\Psi(\Q,g).
\end{equation}
For $g\in C_b$, define
$
X_g
=
\exp(g-c(g))$.
Then $X_g$ is continuous and bounded from above and below. Moreover,
\[
\sup_{\P\in\Pcal}\E_\P[X_g]
=
e^{-c(g)}
\sup_{\P\in\Pcal}\E_\P[e^g]
=
1,
\]
so $X_g\in\Ecal_{bc}$. Since
$
\E_\Q[\log X_g]=
\Psi(\Q,g)$,
we obtain
$
\Grow_{bc}
\ge
\sup_{g\in C_b}
\inf_{\Q\in\Qcal}
\Psi(\Q,g)$.
Conversely, let $X\in\Ecal_{bc}$ and set $g=\log X$. Then $g\in C_b$ and
$
c(g)
=
\log\sup_{\P\in\Pcal}\E_\P[X]
\le 0$.
Hence, for every $\Q\in\Qcal$,
\[
\E_\Q[\log X]
=
\E_\Q[g]
\le
\E_\Q[g]-c(g)
=
\Psi(\Q,g).
\]
It follows that
$
\inf_{\Q\in\Qcal}
\E_\Q[\log X]
\le
\sup_{h\in C_b}
\inf_{\Q\in\Qcal}
\Psi(\Q,h)$. 
Taking the supremum over $X\in\Ecal_{bc}$ gives the reverse inequality in
\eqref{eq:continuous-payoff-reduction}.

For fixed $g\in C_b$, the map
$\Q\mapsto\Psi(\Q,g)$ is affine and weakly continuous. For fixed
$\Q\in\Qcal$, the map $g\mapsto\Psi(\Q,g)$ is concave and continuous.
Since $\Qcal$ is convex and weakly compact, Sion's theorem gives
\[
\sup_{g\in C_b}\inf_{\Q\in\Qcal}\Psi(\Q,g)
=
\inf_{\Q\in\Qcal}\sup_{g\in C_b}\Psi(\Q,g).
\]
Combining this with \eqref{eq:continuous-payoff-reduction}, we obtain
\begin{align*}
\Grow_{bc}
&=
\inf_{\Q\in\Qcal}
\sup_{g\in C_b}
\left(
\E_\Q[g]
-
\log\sup_{\P\in\Pcal}\E_\P[e^g]
\right) 
=
\inf_{\Q\in\Qcal}\inf_{\P\in\Pcal}H(\Q\mid\P),
\end{align*}
where the last equality follows from
Lemma~\ref{lem:Donsker-Varadhan-convex-compact}. 
Noting that relative entropy is jointly
 weakly lower semicontinuous and
$\Qcal \times \Pcal$ is weakly compact, we conclude that the infimum is attained. 
\end{proof}

The following corollary gives a simple criterion for when an e-variable witnesses strong duality. Its proof is immediate and hence omitted.

\begin{corollary}\label{cor:attainment}
  Assume either the setup of Theorem~\ref{thm:setwise} or of Theorem~\ref{thm:nondominated-compact-convex}, and let  $(\P^*,\Q^*)$ denote any pair that achieves the right-hand side of the strong duality. If $\Grow < \infty$, then $(\P^*_a,\Q^*)$ is a joint information projection. (Here, $\P^*_a$ is the absolutely continuous part of $\P^*$ with respect to $\Q^*$.) If we further assume that Assumption~\ref{ass:composite-P} holds for the pair $(\P^*_a,\Q^*)$ , then Theorem~\ref{T:260608} implies that the e-variable $X^* = d\Q^*/d\P^*_a$ achieves the supremum in $\Grow$. 
\end{corollary}

We also recall Example~\ref{eg:jipr-likelihoodratio-fail} to remind us that even with a simple null and setwise and weakly compact and convex $\Qcal$, strong duality may hold, but the likelihood ratio of the JIPr pair can have an arbitrarily bad worst-case e-power, and no e-variable witnesses the strong duality; see also Example~\ref{ex:260621}. We conclude with an example frequently encountered in the recent e-variable literature.

\begin{example}[Bounded means]\label{eg:bounded-mean}
    Let $\Omega=[0,1]$ and let $Z(\omega)=\omega$ denote the coordinate variable on $[0,1]$. Set
     $\Pcal=\{\P: \E_{\P}[Z] \leq a \}$ and $\Qcal=\{\Q: \E_{\Q}[Z] \geq b \}$, for some $0 < a<b < 1$.  \cite{larsson2026testing} show that every e-variable is bounded by one of the form \(
    1+\lambda(Z-a)
    \)
    for $\lambda \in [0,1/a]$, and is bounded by $1/a$. Therefore, $\Grow=\Grow_b < \infty$ is not surprising. Since $\Pcal$ and $\Qcal$ are both convex and weakly compact,  Corollary~\ref{cor:attainment} applies to conclude that a JIPr pair $(\P^*,\Q^*)$ exists (in this case, uniquely): it is in fact just the pair of Bernoulli distributions with parameters $a,b$. 
    The continuous e-variable obtained by picking $\lambda^* = (b-a)/(a(1-a))$ is easily checked to be a $\P^*$-version of their likelihood ratio.
    
    Assumption~\ref{ass:composite-P} can be verified as follows. Since $\P^* = (1-a)\delta_0+a\delta_1$, $\inf_{\RR \in \Qcal} H(\RR \mid \P^*)$ is achieved by $\Q^*$. (Indeed, any $\RR$ not supported on $\{0,1\}$ has infinite relative entropy, so we only check amongst Bernoullis.) So we need to check that for every $\Q \in \Qcal$, 
    \[
        \E_\Q[\log X^*] \geq H(\Q^* \mid \P^*) = b \log \frac{b}{a} + (1-b)\log\frac{1-b}{1-a}.
    \]
    Denote $\phi(\omega) = \log X^*(\omega) = \log(1+ \lambda^*(\omega-a))$. Since $\phi$ is concave on $[0,1]$, it lies above its chord between 0 and 1: $\phi(\omega) \geq (1-\omega)\phi(0) + \omega\phi(1)$. So for any $\Q \in \Qcal$, $\E_\Q[\phi(Z)] \geq \E_\Q[1-Z]\phi(0) + \E_\Q[Z]\phi(1)$. Since $\E_\Q[Z] \geq b$ and $\phi(1) > \phi(0)$, we get $\E_\Q[\log X^*] \geq (1-b)\phi(0) + b\phi(1) = H(\Q^* \mid \P^*)$, as required.
    
    The GROW optimality of $1+\lambda^*(Z-a)$ thus follows from Corollary~\ref{cor:attainment}. \cite{arnold2026optimal} proved the GROW optimality of this e-variable as well, but it is reassuring that we can derive the same result as a special case of our general theory. To emphasize this point, we present similar calculations for a slightly more complicated problem in Appendix~\ref{app:bounded-second-moment}, where unbounded data have bounded second moment, and the null and alternative classes differ in their means.
\end{example}

Other natural statistical examples that are convex and weakly compact include relative entropy balls $\Pcal=\{\P: H(\P \mid \P_0) \leq c_0\}$ and $\Qcal=\{\Q: H(\Q \mid \Q_0) \leq c_1\}$, for some $\P_0,\Q_0\in\Mone$ and $c_0,c_1 > 0$, or Hellinger balls on compact Polish $\Omega$.

\subsection{When the sample space is finite}\label{subsec:finite-Omega}

Below, we consider the case when $\Omega=\{1,\ldots,d\}$, and identify probability measures on
$\Omega$ with vectors in the simplex
\[
\Delta_d=\left\{p\in[0,1]^d:\sum_{i=1}^d p_i=1\right\}.
\]

\begin{theorem}[Finite sample spaces with no null coordinates] \label{T:260619}
Let $\Pcal,\Qcal\subseteq\Delta_d$ be convex and compact
sets. Assume that $\Pcal$ is supported on all coordinates, meaning that $\sup_{\P \in \Pcal} \P_i > 0$ for all $i \in \{1, \ldots, d\}$. 
Then $\Ecal=\Ecal_b$ and 
\[
\Grow=\Grow_b
=
\max_{X\in\Ecal_b}
\inf_{\Q\in\Qcal}
\E_\Q[\log X]
=
\min_{\Q\in\Qcal}\min_{\P\in\Pcal}H(\Q \mid \P) < \infty,
\]
meaning that the GROW value is finite and is attained by some bounded
e-variable $X^*\in\Ecal_b$. Moreover, if $(\P^*,\Q^*)\in\Pcal\times\Qcal$ minimizes the right-hand side, then
$
X^*={d\Q^*}/{d\P^*}$, $\P^*\text{-a.s.}$
\end{theorem}

\begin{proof}
Any $X \in \Ecal$ takes only finitely many values, all finite due to the support assumption on $\Pcal$. This shows 
$\Ecal=\Ecal_b$, so $\Grow = \Grow_b$.
Next, since $\Omega$ is finite, weak convergence on $\Delta_d$ is the same as coordinatewise, hence Euclidean, convergence. Thus the assumed compactness of $\Pcal$ and $\Qcal$ implies weak compactness. Moreover, weak-$*$ closure agrees with Euclidean closure in finite dimension, and therefore \[ \cconv^*(\Pcal)=\Pcal, \qquad \cconv^*(\Qcal)=\Qcal. \]
Theorem~\ref{T_strong_duality} then gives
\begin{align} \label{eq:260623}
\Grow = \Grow_b
=
\min_{\Q\in\Qcal}\min_{\P\in\Pcal}H(\Q \mid \P).
\end{align}
To show that the right-hand side above is finite, for each $i$ choose
$\P^{(i)}\in\Pcal$ with $\P_i^{(i)}>0$, and define
$
\overline \P=\frac1d\sum_{i=1}^d \P^{(i)} \in \Pcal$,
by convexity. Then every $\Q\in\Qcal$ satisfies $\Q\ll \overline \P$, so
$H(\Q\mid \overline \P)<\infty$.

Finally, we prove that the supremum over e-variables is attained. Note that \[ \Ecal \subseteq \prod_{i=1}^d \left[0,\frac{1}{\P_i^{(i)}}\right]. \] Moreover, $\Ecal$ is closed, since it is defined by the closed linear inequalities \[ \sum_{i=1}^d \P_iX_i\le1, \qquad \P\in\Pcal. \] Thus $\Ecal$ is compact.
For fixed $\Q\in\Qcal$, the map
\(
X\mapsto \E_\Q[\log X]
\)
is upper semicontinuous, hence so is
$
X\mapsto
\inf_{\Q\in\Qcal}
\E_\Q[\log X]
$.
Therefore this objective attains its maximum on the compact set $\Ecal$.

Let $(\P^*,\Q^*)\in\Pcal\times\Qcal$ now minimize the right-hand side of \eqref{eq:260623}, and let
$X^*\in\Ecal$ be any GROW optimizer. By Proposition~\ref{prop:klinf-infkl} we have 
 \[     
    \min_{\Q \in \Qcal} \min_{\P \in \Peff} H(\Q \mid \P)  = H(\Q^* \mid \P^*).
    \]
Applying Proposition~\ref{prop:version-of-LR} then yields the final claim.
\end{proof}

\section{Additional examples}\label{sec:examples}

\subsection{$\Grow > \Grow_b$ with singleton $\Pcal = \{\P\}$ and countable $\Qcal$ on $\N$}

While our main Theorem~\ref{T_strong_duality} establishes strong duality for $\Ecal_b$ without imposing any assumptions on $\Pcal$ or $\Qcal$, the strong duality results for $\Ecal$ in Section~\ref{sec:grow} require additional restrictions on these families. A natural question is whether such restrictions are necessary. We answer this in part by presenting a small modification of Example~3.15 in \cite{ramdas2024hypothesis}, which shows that $\Grow>\Grow_b$ can occur. Notably, the example only involves a singleton null $\Pcal=\{\P\}$ (which is hence compact in any reasonable topology), and a countable alternative $\Qcal=\{\Q_n:n\in\N\}$ on $\N$. In particular,  these measures admit a common dominating measure. Thus, even in a very regular setting, the equality $\Grow=\Grow_b$ may fail.

\begin{example}\label{ex:example-from-ramdas-wang}
Let $\Omega=\N_0=\N\cup\{0\}$, let $\P=(1/2)\delta_0+(1/2)\delta_1$, where $\delta_x$ denotes the point mass at $x$, and let $\Qcal=\{\Q_n:n\in\N\}$, where \[ \Q_n=\frac1n\delta_n+\left(1-\frac1n\right)\delta_0. \] Clearly, the family admits a common dominating reference measure. Moreover, \[ \inf_{\Q\in\Qcal}H(\Q\mid\P)=H(\Q_1\mid\P)=\log 2<\infty, \] since $\Q_1=\delta_1$, whereas $H(\Q_n\mid\P)=\infty$ for all $n\ge2$. The same remains true after passing to $\conv(\Qcal)$, and hence the (generalized) information projection of $\P$ onto $\conv(\Qcal)$ equals $\Q_1$.
We further have the following observations.
    \begin{enumerate}
\item[(i)]
For $\varepsilon\in (0,1)$, define
\[
X^\varepsilon(n)
=
\exp\left(n\log(2-\varepsilon)-(n-1)\log\varepsilon\right),
\qquad n\in\N_0.
\]
Then $X^\varepsilon(0)=\varepsilon$ and $X^\varepsilon(1)=2-\varepsilon$, hence
$
\E_\P[X^\varepsilon]
=
1$ and $X^\varepsilon$ is an e-variable for $\P$. Moreover, for every $n\in \mathbb N$,
\[
\E_{\Q_n}[\log X^\varepsilon]
=
\frac1n\left(n\log(2-\varepsilon)-(n-1)\log\varepsilon\right)
+
\left(1-\frac1n\right)\log\varepsilon
=
\log(2-\varepsilon).
\]
Therefore
$
\inf_{\Q\in\Qcal}\E_\Q[\log X^\varepsilon]
=
\log(2-\varepsilon)$,
which tends to $\log 2$ as $\varepsilon\downarrow0$. By weak duality (recall \eqref{weak_duality_general}), this is the
largest possible value. Hence $\Grow=\log 2$, and the conclusion of Theorem~\ref{thm:singleton-P} holds, although its assumption fails. 
    \item [(ii)]
    There does not exist any  bounded e-variable $X$ for which $\E_\Q[X] > 1$ for all $\Q \in \Qcal$. Indeed, meeting this requirement for $\Q_1$ alone would imply $X(1)>1$. To meet this requirement for $\Q_n$, we would need $(1/n) X(n) + (1-1/n) X(0)>1$ for all $n\in \N$.
    Since $X$ is bounded, this then implies $X(0)\ge 1$. 
    Together with $X(1)>1$, we have $\E_\P[X]>1$, conflicting the fact that $X$ is an e-variable for $\P$. 

     \item [(iii)]
    Point (ii) above and Jensen's inequality then imply that there is no bounded e-variable $X$ for which $\inf_{\Q\in\Qcal}\E_\Q[\log X] > 0$. 
     This means that $\Grow_b=0$.

\item [(iv)]    Note that here $\P$ does not lie in $\conv(\Qcal)$, but it does lie within its TV-closed convex hull (hence also within its weak-$*$ closure). To see this, note that by sending $n\uparrow\infty$, $\Q_n$ converges to $\delta_0$ in total variation. Averaging $\delta_0$ with $\Q_1$ yields $\P$. Thus the information projection of $\P$ onto $\co(\Qcal)$, which equals $\Q_1$, differs from its information projection onto the TV-closure of $\co(\Qcal)$, which equals $\P$ itself. This reinforces a point made by~\cite[Example~3.1]{csiszar1984sanov} that the generalized I-projection of $\P$ onto $\co(\Qcal)$ cannot in general be seen as the I-projection onto some closure of $\co(\Qcal)$.
\end{enumerate}
\end{example}

The above example shows that strong duality alone does not imply $\Grow=\Grow_b$.
Indeed, strong duality holds for $\Grow$ by (i), while Theorem~\ref{T_strong_duality} gives the corresponding strong-duality representation for $\Grow_b$; nevertheless, in this example $\Grow>\Grow_b$. This failure occurs despite substantial structure: the null class is a singleton, the families admit a common dominating reference measure, and the relevant information projection is explicit.

\subsection{Extension of Example~\ref{ex:singletonP} that maintains $\Q \ll \P$ for all $\Q \in \Qcal$}

We now provide the promised extension of Example~\ref{ex:singletonP}, which additionally maintains $\Q \ll \P$ for all $\Q \in \Qcal$. Here, an IPr exists, the supremum in $\Grow$ is achieved, but strong duality fails.

\begin{example}\label{ex:singletonP-continued}
As in Example~\ref{ex:singletonP}, we consider 
 $\Omega=[0,1]$, equipped with the Borel sigma algebra,  $\P=\mathsf U$, and the probability measure $\Q_0$ given by 
 $
{d\Q_0}/{d\P}(\omega)=2\omega.
$
Let
\[
\mathcal I_\infty
=
\{\Q\ll \P : H(\Q\mid \P)=\infty\},
\]
and define
\[
\Qcal
=
\{\Q_0\}
\cup
\bigl\{
(1-t)\Q_0+t\RR: t \in (0,1], \RR\in\mathcal I_\infty\}.
\]
Note that $\mathcal I_\infty$ is convex, since positive mixtures of elements of $\mathcal I_\infty$ remain in $\mathcal I_\infty$: if $H(\Q\mid\P)=\infty$, then  $H((1-t)\RR+t\Q\mid\P)=\infty$ for every $\RR\ll\P$ and $t\in(0,1]$. Thus $\Qcal$ is also convex.
By the same token, the only finite-entropy element of $\Qcal$ is $\Q_0$, so
\begin{align} \label{eq:260603.1}
\inf_{\Q\in\Qcal} H(\Q\mid \P)
=
H(\Q_0\mid \P)
=
\log 2-\frac12.
\end{align}
Moreover, $\Q^{\mathrm{GIPr}}=\Q_0$. 
Indeed, the Pythagorean inequality \eqref{eq:pythagorean} holds with equality at $\Q_0$, while for every other element of $\Qcal$ the left-hand side is infinite.

Now, let $X$ be any e-variable for $\P$. Define
\[
A=\{\omega:X(\omega)\le 1\}.
\]
As in Example~\ref{ex:singletonP} we have $\P(A)>0$.
Since $\P$ is nonatomic, there exists a probability measure $\Q_A\ll \P$ supported on $A$ such that $
H(\Q_A\mid \P)=\infty$. 
Consequently $\Q_A\in\Qcal$. Since $X\le 1$ $\Q_A$-almost surely,
\(
\E_{\Q_A}[\log X]
\le 0,
\)
and therefore
\[
\inf_{\Q\in\Qcal}\E_\Q[\log X]
\le
\E_{\Q_A}[\log X]
\le 0.
\]
Since this holds for every e-variable $X$,
\(
\sup_{X\in\Ecal}
\inf_{\Q\in\Qcal}
\E_\Q[\log X]
\le 0.
\)
On the other hand, choosing $X\equiv 1$ yields
\(
\inf_{\Q\in\Qcal}\E_\Q[\log X]
=
0.
\)
Hence
\[
\sup_{X\in\Ecal}
\inf_{\Q\in\Qcal}
\E_\Q[\log X]
=
0.
\]
Combining this with \eqref{eq:260603.1},
\[
\sup_{X\in\Ecal}
\inf_{\Q\in\Qcal}
\E_\Q[\log X]
=
0
<
\log 2-\frac12
=
\inf_{\Q\in\Qcal}
H(\Q\mid \P).
\]
\end{example}

The above example violates strong duality in Theorem~\ref{thm:singleton-P}, but (as expected) the strong duality in Theorem~\ref{T_strong_duality} still holds. 
To see this, first note that
\(
\sup_{X\in\Ecal}
\inf_{\Q\in\Qcal}
\E_\Q[\log X]
=
0
\)
implies immediately that \(
\sup_{X\in\Ecal_b}
\inf_{\Q\in\Qcal}
\E_\Q[\log X]
=
0
\), both achieved by $X\equiv1$.
But the weak-$*$ closure $\cconv^*(\Qcal)$ contains $\P$, so the right-hand side of Theorem~\ref{T_strong_duality} also equals zero.

\subsection{Strong duality for singleton $\Pcal=\{\P\}$ holds, but $\Grow > H(\Q^* \mid \P)$ for GIPr $\Q^*$}\label{subsec:csiszar}

We now revisit Example~3.2 from~\cite{csiszar1984sanov}, which makes the point that 
$\inf_{\Q\in\Qcal}H(\Q\mid \P) > H(\Q^* \mid \P)$ for the GIPr $\Q^*$. We now complete the duality picture for this example. We show that strong duality for $\Grow$ holds,
but despite having a singleton $\Pcal=\{\P\}$, $\Grow$ may not equal $H(\Q \mid \P)$ for any $\Q \in \Qcal$, nor equal $H(\Q^* \mid \P)$. 

Example~\ref{eg:jipr-definitions} already shows that it is possible to satisfy conditions (i) and (ii) in the JIPr Definition~\ref{def:JIPr} without satisfying condition (iii). The following example goes one step further. It shows that even if Assumption~\ref{ass:singleton-P} is satisfied, it is possible to satisfy conditions (i) and (ii) in the JIPr definition~\ref{def:JIPr} without satisfying condition (iii).

\begin{example}\label{eg:csiszar}
Let $\Omega = [0, \infty)$, equipped with its Borel sigma algebra.
    Let $$\Qcal_a = \{\Q \in \Mcal_1: \int s d\Q(s)  \geq a\}$$ for some $a > 2/3$, which is convex.  
    Let $\P$ be defined by the density $p(s) = \exp(-s)(s+4)/(s+1)^4$ for $s\geq 0$, which has mean about $0.298 < 2/3$, and thus lies outside $\Qcal_a$. Csisz\'ar proves that the GIPr  of $\P$ onto $\Qcal_a$ is a distribution $\Q^*$ with density $q^*(s) = \frac23 (s+4)/(s+1)^4$  for $s\geq 0$, which remarkably does not depend on $a$, and its mean is $2/3$, so  $\Q^* \notin \Qcal_a$. We have the following facts.
    \begin{enumerate}
        \item For any $\Q\in \Qcal_a$ with $\Q\ll \Q^*$, 
        the chain rule gives \[  H(\Q\mid \P) = H(\Q\mid \Q^*)+ \E_\Q\left[\log\frac{d\Q^*}{d\P}\right] = H(\Q\mid \Q^*)+ \int s d \Q(s) +\log\frac23 \ge a+\log\frac23.  \] If $\Q\not\ll \Q^*$, since $\Q^* \sim \P$, we have $H(\Q\mid \P)=\infty$, and the same lower bound is trivial. Put together, 
         \( \inf_{\Q\in\Qcal_a}H(\Q\mid \P) \ge a+\log\frac23. \) In fact, equality holds, as justified in Appendix~\ref{app:sec5}.
        \item 
Since
\[
    \frac{q^*(s)}{p(s)}
    =
    \frac{\frac23 (s+4)/(s+1)^4}
         {e^{-s}(s+4)/(s+1)^4}
    =
    \frac23 e^s,
    \qquad s\geq 0,
\]
we choose the Borel version of the Radon--Nikodym derivative
\(d\Q^*/d\P\) given by
$
    X^*(s)=\frac23 e^s$.
 Then
$
    X^*=\frac{d\Q^*}{d\P}$, $\P\text{-a.s.}$, hence \(X^*\in\Ecal\).
 For every $\Q \in \Qcal_a$, $$\E_\Q[\log X^*] = \int s d \Q(s)+\log\frac23 \ge a+\log\frac23 = \inf_{\RR\in\Qcal_a}H(\RR\mid\P) .$$ 
 Thus,  Assumption~\ref{ass:singleton-P} is satisfied, and, by Theorem~\ref{thm:singleton-P}, strong duality holds and $X^*$ is GROW optimal.
        
        \item $H(\Q^* \mid \P) =2/3 - \log(3/2) < a + \log(2/3)$.
    \end{enumerate}
    To summarize, we have 
    \[
a + \log\frac{2}{3}= \sup_{X\in\mathcal{E}}\inf_{\Q\in\mathcal{Q}_a} \E_{\Q}[\log X] = \inf_{\Q \in \Qcal_a} H(\Q \mid \P)   > H(\Q^* \mid \P),
    \]
    meaning that strong duality holds, and the left-hand side supremum is achieved by the e-variable $X^* = d\Q^*/d\P$. Furthermore,  $\inf_{\Q\in\mathcal{Q}_a} \E_{\Q}[\log X^*]$ is not achieved by $\Q^*$, but it is achieved by any $\Q \in \Qcal_a$ with mean equal to $a$.
    Finally, $(\P,\Q^*)$ also forms a JIPr pair satisfying conditions (i),(ii) but not condition (iii) in Definition~\ref{def:JIPr}. 
\end{example}

\subsection{Gaussian mean shift with heavy-tailed outliers}

We now formalize and generalize Remark~\ref{rem:mixture-epower}, which allows us to extend Example~\ref{eg:jipr-likelihoodratio-fail} to a more statistically meaningful context.

\begin{theorem}[Shielding theorem] \label{T:shielding}
Suppose that
$\Qcal$ is partitioned as
\(
\Qcal=\Qcal_I\cup \Qcal_\infty,
\)
where $\Qcal_I$ is nonempty. Assume that strong duality holds over
$\Qcal_I$ with a finite value, i.e.,
\begin{equation*}
I
=
\sup_{X\in\Ecal}
\inf_{\Q\in\Qcal_I}\E_\Q[\log X]
=
\inf_{\Q\in\Qcal_I}
\inf_{\P\in\Peff}H(\Q \mid \P)
<\infty.
\end{equation*}
Assume further that there exists $Y\in\Ecal$ such that
\begin{equation}\label{eq:shield-assumption}
\E_\Q[\log Y]=\infty
\qquad
\text{for every }\Q\in\Qcal_\infty.
\end{equation}
Then strong duality holds over the full alternative $\Qcal$, i.e.,
\begin{equation}\label{eq:shielded-strong-duality}
I = \sup_{X\in\Ecal}
\inf_{\Q\in\Qcal}\E_\Q[\log X]
=
\inf_{\Q\in\Qcal}
\inf_{\P\in\Peff}H(\Q\mid \P).
\end{equation}
\end{theorem}
Appendix~\ref{app:sec5} provides the proof of Theorem~\ref{T:shielding}.

\begin{example}[Gaussian mean shift with heavy-tailed outliers]\label{ex:260621}
Let $\Omega=\mathbb R$ and let 
\(
\Pcal=\{\P\}\) with \(
\P=N(0,1)
\).
Fix $\theta>0$ and let
\(
\Q_0=N(\theta,1).
\)
The likelihood ratio of $\Q_0$ with respect to $\P$ is
\(
L(\omega)=
\exp\left(\theta \omega-\frac{\theta^2}{2}\right),
\)
which is an e-variable with
\(
\E_{\Q_0}[\log L]
= H(\Q_0\mid \P)
= \frac{\theta^2}{2}.
\)
Thus strong duality holds for the singleton alternative
\(
\Qcal_I=\{\Q_0\},
\)
with value
\(
I=\frac{\theta^2}{2}
\).
Now let $\RR$ be a centered Student-$t_\kappa$ distribution with
$1<\kappa<2$.  For $t\in[0,1]$,
define
\(
\Q_t=(1-t)\Q_0+t\RR,
\)
and set
\[
\Qcal
=
\{\Q_t:0\le t\le1\},
\qquad
\Qcal_\infty
=
\{\Q_t:0<t\le1\}.
\]
This gives the partition
\(
\Qcal=\Qcal_I\cup\Qcal_\infty,
\)
and we note that $\Qcal$ is both setwise and weakly compact, so Theorem~\ref{thm:setwise} already gives strong duality, but we show below that it is not attained by any e-variable, and shielding helps verify this strong duality.

Choose now $a\in(0,1/2)$ and define the random variable
\[
Y_a(\omega)
=
\sqrt{1-2a}\exp(a\omega^2).
\]
Since 
\(
\E_\P[Y_a]=1
\),
we get $Y_a \in \Ecal$. Moreover,
$
\log Y_a(\omega)
=
\frac12\log(1-2a)+a\omega^2$.
Therefore, for every $t>0$,
\[
\E_{\Q_t}[\log Y_a]
=
(1-t)\E_{\Q_0}[\log Y_a]+t \E_\RR[\log Y_a]
=
\infty,
\]
because $\RR$ has infinite second moment. Thus $Y_a$ is a shielding e-variable for
$\Qcal_\infty$.
By Theorem~\ref{T:shielding}, strong duality holds:
\[
\frac{\theta^2}{2} = \sup_{X\in\Ecal}
\inf_{\Q\in\Qcal}\E_\Q[\log X]
=
\inf_{\Q\in\Qcal}H(\Q\mid \P).
\]
Proposition~\ref{prop:version-of-LR} implies that this strong duality is not attained by any e-variable. Indeed, if it was obtained, such e-variable would have to equal $d\Q_0/d\P$, $\P$-a.s., and hence also $\RR$-a.s., but $\E_\RR[\log d\Q_0/d\P] = \E_\RR[\theta Z - \theta^2/2] = -\theta^2/2 < I$, where $Z(\omega) = \omega$.
\end{example}

\subsection{Finite bounded moment inequalities}\label{SS:moments}

We next present an example of a natural class that was recently studied by~\cite{larsson2026testing}; there it served the role of $\Pcal$, and below it serves the role of $\Qcal$ in Theorem~\ref{thm:singleton-P}.

\begin{example}
\label{ex:finite-moment-inequalities}
Let $\P\in \mathcal M_1$, let $f_1,\ldots,f_K\in \Bb$, and define
\[
    \Qcal
    =
    \left\{
        \Q\in \Mcal_1:
        \E_\Q[f_k]\le0,\quad k=1,\ldots,K
    \right\}.
\]
I-projections under finitely many linear inequality constraints were
studied by \cite{dykstraWollan1987}; see also \cite{bhattacharyaDykstra1995}.
Assume the strict-feasibility condition that there exists
$\RR \in\Qcal$ such that $\RR \ll \P$ and
\[
    H(\RR\mid \P)<\infty,
    \qquad
    \E_{\RR}[f_k]<0,\quad k=1,\ldots,K.
\]
For example, this holds if there exist $A\in\mathcal F$ and $\eta>0$
such that
\[
    \P(A)>0,
    \qquad
    f_k\le-\eta\ \text{on }A,\quad k=1,\ldots,K,
\]
by taking $\RR=\P(\,\cdot\mid A)$.
Define
\[
    \psi(\lambda)
    =
    \log \E_\P\!\left[
        \exp\left(-\sum_{k=1}^K\lambda_kf_k\right)
    \right],
    \qquad
    \lambda\in[0,\infty)^K.
\]
Since the functions $f_k$ are bounded, $\psi$ is finite, convex, and continuously
differentiable. Moreover, since
\[
    \eta
    =
    \min_{1\le k\le K}(-\E_{\RR}[f_k])>0,
\]
 the Donsker--Varadhan inequality (Lemma~\ref{lem:extended-entropy-sec}) gives
\[
\begin{aligned}
    \psi(\lambda)
    &\ge
    -H(\RR\mid \P)
    -\sum_{k=1}^K\lambda_k\E_{\RR}[f_k] \ge
    -H(\RR\mid \P)+\eta\|\lambda\|_1.
\end{aligned}
\]
Thus $\psi$ attains its
minimum at some $\lambda^*\in[0,\infty)^K$.
Set
\[
    X^*
    =
    \exp\left(
        -\sum_{k=1}^K\lambda_k^*f_k-\psi(\lambda^*)
    \right)
\]
Then $\E_\P[X^*]=1$.  
Define next $\Q^*$ by 
    $\frac{d\Q^*}{d\P}=X^*$.
The first-order optimality conditions at $\lambda^*$ give
\[
    \E_{\Q^*}[f_k] = -\frac{\partial\psi}{\partial\lambda_k}(\lambda^*) \le0, 
        \qquad
    \lambda_k^*\E_{\Q^*}[f_k]=-\lambda_k^* \frac{\partial\psi}{\partial\lambda_k}(\lambda^*)= 0,
    \qquad k=1,\ldots,K.
\]

Hence $\Q^*\in\Qcal$ and 
\[
\begin{aligned}
    H(\Q^*\mid \P)
    &=
    -\sum_{k=1}^K\lambda_k^*\E_{\Q^*}[f_k]
    -\psi(\lambda^*) =
    -\psi(\lambda^*).
\end{aligned}
\]
For every $\Q\in\Qcal$,
\[
\begin{aligned}
    \E_\Q[\log X^*]
    &=
    -\sum_{k=1}^K\lambda_k^* \E_\Q[f_k]
    -\psi(\lambda^*) \ge
    -\psi(\lambda^*)
    =
    H(\Q^*\mid \P).
\end{aligned}
\]
Combining the above with weak duality, we get
\[
    H(\Q\mid \P)
    \ge
    \E_\Q[\log X^*]
    \ge
    H(\Q^*\mid \P),
    \qquad \Q\in\Qcal.
\]
Thus $\Q^*$ is the IPr of $\P$ onto $\Qcal$, and the
preceding display verifies Assumption~\ref{ass:singleton-P}. Consequently,
\[
    \Grow
    =
    \inf_{\Q\in\Qcal}H(\Q\mid \P)
    =
    H(\Q^*\mid \P)
    =
    -\psi(\lambda^*),
\]
and $X^*=d\Q^*/d\P$ is the GROW e-variable.
\end{example}

\subsection{$\Grow_{b} > \Grow_{bc}$ can hold in general}

In the context of Subsection~\ref{subsec:grow-weakly-compact}, $\Grow_b \geq \Grow_{bc}$ holds always. We provide an example where $\Grow_{b} > \Grow_{bc}$ in order to show that some assumptions on $\Qcal$ are needed for $\Grow_{b} = \Grow_{bc}$ to hold.

\begin{example}\label{ex:arbitrary-Q-counterexample}
Let $\Omega=[0,1]$ with its Borel sigma algebra, set $D=\mathbb Q\cap[0,1]$, and let $\Pcal=\{\mathsf U\}$ and $\Qcal=\{\delta_q:q\in D\}$. For $N\in \N$, define $X_N=N\1_D+\1_{D^c} \in \Ecal_b$. Then 
\[ \sup_{X\in\Ecal_b}\inf_{\Q\in\Qcal} \E_\Q[\log X] =\infty. \]

On the other hand, let $X\in\Ecal_{bc}$. Then $X$ is continuous, nonnegative, and satisfies $\mathbb E_{\mathsf U}[X]\le1$. Therefore $\inf_{\omega\in[0,1]}X(\omega)\le1$. Since $D$ is dense in $[0,1]$ and $X$ is continuous, \[ \inf_{q\in D}X(q) = \inf_{\omega\in[0,1]}X(\omega) \le1. \] Consequently, \[ \inf_{\Q\in\Qcal}\E_\Q[\log X] = \inf_{q\in D}\log X(q) \le0. \] Since the constant function $X\equiv1$ belongs to $\Ecal_{bc}$, we obtain \[ \sup_{X\in\Ecal_{bc}}\inf_{\Q\in\Qcal}\E_\Q[\log X] =0. \] Thus $\Grow_b>\Grow_{bc}$ in general.
\end{example}

\section{Summary}

Theorem~\ref{T_strong_duality} establishes strong duality for the minimax e-power of bounded e-variables, equivalently for the 
\emph{bounded} GROW value $\Grow_b$. It shows that $\Grow_b$ equals the minimum relative entropy between the weak-$*$ closed convex hulls of $\Pcal$ and $\Qcal$, without imposing any restrictions on either family. We also extend this result to the REGROW criterion for arbitrary bounded offsets. We then prove several strong duality results for the \emph{unbounded} GROW value $\Grow$.  Theorem~\ref{thm:singleton-P} treats the case of a singleton null $\Pcal$, improving on previous work. Theorem~\ref{T:260608} extends this result to the case that a joint information projection exists. We then prove strong duality results for unbounded GROW under assumptions such as setwise or weak compactness, that do not require taking weak-$*$ closures. It remains open to what extent such restrictions are necessary. While our examples show that $\Grow$ can be strictly larger than $\Grow_b$, we do not know whether a universal strong duality theorem for $\Grow$ holds for arbitrary $\Pcal$ and $\Qcal$. Indeed, Examples~\ref{ex:singletonP} and \ref{ex:singletonP-continued} demonstrate that a naive entropy duality does not hold without assumptions; it remains open whether some alternative universal dual representation for unbounded GROW exists.

\subsection*{Acknowledgments and Disclosures}
AR acknowledges support from the National Science Foundation under grant DMS-2310718.
ML acknowledges support from the National Science Foundation under grant
NSF DMS-2510965. 
% ChatGPT was used to simplify some proofs and construct some examples, but the authors take responsibility for the correctness of all claims in the paper.

\bibliography{references}
\bibliographystyle{plainnat}

\appendix

\section{Proofs for Section~\ref{sec:prelim}}\label{sec:missing-proofs}

\begin{proof}[Proof of Lemma~\ref{lem:weak-star-jipr-zero-iff-equal}]
The implication $\mu=\nu \Rightarrow H(\nu\mid\mu)=0$ is immediate from the
definition. Conversely, suppose that $H(\nu\mid\mu)=0$ and consider some
 $A\in\mathcal F$, along  with the partition $\pi=\{A,A^c\}$.
Then the finite-dimensional entropy
\[
H_\pi(\nu\mid\mu)
=
\sum_{A\in\pi}\nu(A)\log\frac{\nu(A)}{\mu(A)}
\]
is nonnegative, since it is the relative entropy between the two probability
vectors $(\nu(A), \nu(A^c))$ and $(\mu(A), \mu(A^c))$. Hence,
we must have $H_\pi(\nu\mid\mu)=0$.
The finite-dimensional relative entropy of two probability vectors is zero if
and only if the two vectors are equal. Therefore
\[
\nu(A)=\mu(A)
\qquad\text{and}\qquad
\nu(A^c)=\mu(A^c).
\]
Since $A\in\mathcal F$ was arbitrary, it follows that $\nu=\mu$.
\end{proof}

\begin{proof}[Proof of Lemma~\ref{lem:extended-entropy-sec}] 
Write
\[
F(f)=\int f\,d\nu-\log\int e^f\,d\mu,
\qquad f\in\Bb.
\]

We first show that \(F(f)\le H(\nu\mid\mu)\) for every \(f\in\Bb\).
Suppose first that \(f\) is simple, say
$
f=\sum_{i=1}^n c_i 1_{A_i}
$
for some finite measurable partition \(\pi=\{A_1,\dots,A_n\}\). Set
\[
q_i=\nu(A_i),\qquad p_i=\mu(A_i),\qquad \kappa=\sum_{i=1}^n p_i e^{c_i}.
\]
If \(q_i>0\) and \(p_i=0\) for some \(i\), then \(H_\pi(\nu\mid\mu)=\infty\), so there is nothing to prove. Otherwise $\kappa>0$ and we define
$
r_i={p_i e^{c_i}}/{\kappa}
$.
Then \((r_i)\) is a probability vector, and recalling convention \eqref{eq_q_log_q_by_p_convention},
\[
F(f) = \sum_{i=1}^n q_i (c_i-\log \kappa)
=
\sum_{i=1}^n q_i \log\frac{r_i}{p_i}
=
\sum_{i=1}^n q_i \log\frac{q_i}{p_i}
-
\sum_{i=1}^n q_i \log\frac{q_i}{r_i}.
\]
Since
$
\sum_{i=1}^n q_i \log({q_i}/{r_i})\ge 0$,
it follows that
\[
F(f)\le \sum_{i=1}^n q_i \log\frac{q_i}{p_i}=H_\pi(\nu\mid\mu)\le H(\nu\mid\mu).
\]

Now let \(f\in\Bb\) be arbitrary. Choose simple functions \(f_m\in\Bb\) such that
$
\|f_m-f\|_\infty\to 0$.
Then
\[
\left|\int f_m\,d\nu-\int f\,d\nu\right|\le \|f_m-f\|_\infty\to 0.
\]
Also, if \(\varepsilon_m=\|f_m-f\|_\infty\), then
\[
e^{-\varepsilon_m}\int e^{f_m}\,d\mu
\le
\int e^f\,d\mu
\le
e^{\varepsilon_m}\int e^{f_m}\,d\mu,
\]
so
\[
\left|\log\int e^f\,d\mu-\log\int e^{f_m}\,d\mu\right|
\le \varepsilon_m\to 0.
\]
Therefore \(F(f_m)\to F(f)\). Since \(F(f_m)\le H(\nu\mid\mu)\) for every \(m\), we conclude that
\[
\sup_{f\in\Bb}F(f)\le H(\nu\mid\mu).
\]

For the reverse inequality, fix a finite measurable partition
$
\pi=\{A_1,\dots,A_n\}$, 
and again write
$
q_i=\nu(A_i)$ and $p_i=\mu(A_i)$.
If \(q_i>0\) and \(p_i=0\) for some \(i\), then \(H_\pi(\nu\mid\mu)=\infty\). In that case, letting
$
f_m=m\, \1_{A_i}$, 
we obtain
\[
F(f_m)\ge m q_i-\log\int e^{f_m}\,d\mu = m q_i - \log \mu(A_i^c) \uparrow \infty.
\]
 Hence
$
\sup_{f\in\Bb}F(f)=\infty=H_\pi(\nu\mid\mu)$.
Now assume instead that \(q_i=0\) whenever \(p_i=0\). For \(m\in\mathbb N\), define
\[
c_i^{(m)}=
\begin{cases}
\log(q_i/p_i), & q_i>0,\\
-m, & q_i=0,
\end{cases}
\qquad
f_m=\sum_{i=1}^n c_i^{(m)} \1_{A_i}\in\Bb.
\]
Then
\[
\int f_m\,d\nu=\sum_{q_i>0} q_i\log\frac{q_i}{p_i},
\]
while
\[
\int e^{f_m}\,d\mu
=
\sum_{q_i>0} p_i\frac{q_i}{p_i}
+
\sum_{q_i=0} p_i e^{-m}
=
1+e^{-m}\sum_{q_i=0}p_i.
\]
Therefore
\[
F(f_m)
=
\sum_{q_i>0} q_i\log\frac{q_i}{p_i}
-
\log\left(1+e^{-m}\sum_{q_i=0}p_i\right)
\to
\sum_{i=1}^n q_i\log\frac{q_i}{p_i}=H_\pi(\nu\mid\mu).
\]
Hence
$
H_\pi(\nu\mid\mu)\le \sup_{f\in\Bb}F(f)$.
Since \(\pi\) was arbitrary, taking the supremum over all finite measurable partitions gives
$
H(\nu\mid\mu)\le \sup_{f\in\Bb}F(f)$, concluding the proof of the first part.

For the second part, if \(\int f^-\,d\nu=\infty\), the claim follows from our convention for extended integrals. Otherwise,
set \(f_n=(f\vee(-n))\wedge n\in \Bb\). From the first part of the lemma, we then have
\[
\int f_n\,d\nu-\log\int e^{f_n}\,d\mu\leq H(\nu\mid \mu).
\]
By monotone convergence applied separately to the positive and negative
parts, \(\int f_n\,d\nu\to\int f\,d\nu\). Moreover,
\[
e^f\wedge e^n\leq e^{f_n}\leq e^f\wedge e^n+e^{-n},
\]
and hence \(\int e^{f_n}\,d\mu\to\int e^f\,d\mu\). Letting \(n\uparrow\infty\)
proves the claim.
\end{proof}

\begin{proof}[Proof of Proposition~\ref{prop:ba-countably-additive-reduction}]
Fix $\mu\in\ba_+$ and let
$\mu=\mu_c+\mu_p$ denote its Yosida--Hewitt decomposition. Since $\mu_c\le \mu$, we immediately have
$
H(\Q\mid\mu)\le H(\Q\mid\mu_c)$. 

We next prove the reverse inequality. Let $
\RR=\Q+\mu_c$.
Since $\RR$ is countably additive and $\mu_p$ is purely finitely additive,
there is no nonzero positive finitely additive measure dominated by both
$\mu_p$ and $\RR$. Equivalently, 
we may choose a sequence $(B_n)\subseteq\mathcal F$ such that
\[
\mu_p(B_n)+\RR(B_n^c)\to 0. 
\]
 Fix now a finite measurable partition
$\pi=\{A_1,\ldots,A_m\}$. For each $n$, consider the finite measurable partition
$
\pi_n
=
\{A_1\cap B_n,\ldots,A_m\cap B_n,B_n^c\}$. 
Then
$
H(\Q\mid\mu)\ge H_{\pi_n}(\Q\mid\mu)$.
For each $i=1,\ldots,m$, we have
$
\Q(A_i\cap B_n)\to \Q(A_i)$,
and
\[
\mu(A_i\cap B_n)
=
\mu_c(A_i\cap B_n)+\mu_p(A_i\cap B_n)
\to
\mu_c(A_i).
\]
 Hence, using \eqref{eq_q_log_q_by_p_convention}, which implies the lower semicontinuity of
$q\log(q/p)$, yields
\[
\liminf_{n\uparrow\infty}
\sum_{i=1}^m
\Q(A_i\cap B_n) 
\log\frac{\Q(A_i\cap B_n)}{\mu(A_i\cap B_n)}
\ge
\sum_{i=1}^m
\Q(A_i)\log\frac{\Q(A_i)}{\mu_c(A_i)}
=
H_\pi(\Q\mid\mu_c).
\]
The remaining atom $B_n^c$ does not affect the lower bound, since
$\Q(B_n^c)\to0$ and
$
\Q(B_n^c)\log({\Q(B_n^c)}/{\mu(B_n^c)})$
has nonnegative limit inferior.
Therefore,
\[
H(\Q\mid\mu)
\ge
\liminf_{n\uparrow\infty} H_{\pi_n}(\Q\mid\mu)
\ge
H_\pi(\Q\mid\mu_c).
\]
Taking the supremum over all finite measurable partitions $\pi$ yields
$
H(\Q\mid\mu)\ge H(\Q\mid\mu_c)$, 
hence the first claim is established.

Now consider $\nu\in\ba_1 \setminus \Mone$. Then continuity from above fails for $\nu$. Hence there exists a nonincreasing sequence $(A_n)\subseteq\mathcal F$ such that
$
A_n\downarrow\varnothing
$
but there exists some $\delta>0$ such that
$\nu(A_n) \geq \delta$ for all $n$. 
On the other hand, since $\P\in\Mplus$ is countably additive and finite, we have
$
\P(A_n)\downarrow0$.
Consider the two-point partitions $\pi_n=\{A_n,A_n^c\}$. Then
\[
H(\nu\mid\P)
\ge
H_{\pi_n}(\nu\mid\P)
\geq
\nu(A_n)\log\frac{\nu(A_n)}{\P(A_n)}
+
\nu(A_n^c)\log\frac{\nu(A_n^c)}{\P(\Omega)}.
\]
The first term now tends to $\infty$, whereas the second term is bounded from below.
This yields the claim.
\end{proof}

\begin{proof}[Proof of Lemma~\ref{lem:E-to-Eb}]
Since $\Ecal_b\subseteq\Ecal$, we only have to argue
$
\sup_{Y\in\Ecal_b}\E_\Q[\log Y]
\geq
\E_\Q[\log X]$ for all $X \in \Ecal$ such that $\E_\Q[(\log X)^-] < \infty$.
We fix such an $X$ and define $X_m=X\wedge m \in \Ecal_b$ for
$m \in\mathbb N$. 
Then by monotone convergence, 
$
\E_\Q[\log X_m]
\uparrow
\E_\Q[\log X]$.
This yields the statement.
\end{proof}

\section{Details of Example~\ref{ex:weak-star-jipr-necessarily-finitely-additive}}\label{app:example}

We first show that the weak-$*$ closed convex hulls of $\Pcal$ and
$\Qcal$ have a common finitely additive element. Since $\ba_1$ is
weak-$*$ compact, the sequence $\P_n$ has a weak-$*$
convergent subnet $\P_{n_\alpha}$, say with limit
$\mu\in\ba_1$. Since $n_\alpha$ is a subnet of the sequence of
integers, we have $n_\alpha\uparrow\infty$. For any $f\in\Bb$,
\[
\int f\,d\Q_{n_\alpha}-\int f\,d\P_{n_\alpha}
=
\frac{1}{n_\alpha}
\sum_{k=2}^{n_\alpha+1} f(k)
-
\frac{1}{n_\alpha}
\sum_{k=1}^{n_\alpha} f(k)
=
\frac{f(n_\alpha+1)-f(1)}{n_\alpha}.
\]
Thus
\[
\left|
\int f\,d\Q_{n_\alpha}-\int f\,d\P_{n_\alpha}
\right|
\le
\frac{2\|f\|_\infty}{n_\alpha}
\downarrow 0.
\]
Since $\P_{n_\alpha}\to\mu$ in $\sigma(\ba,\Bb)$, it follows that also
$\Q_{n_\alpha}\to\mu$ in $\sigma(\ba,\Bb)$. Hence $
\mu\in\cconv^*(\Pcal)\cap\cconv^*(\Qcal)$.

We next show that $\mu$ is not countably additive. Fix $m\in\mathbb N$. Since
$n_\alpha\uparrow\infty$, we have $n_\alpha\ge m$ eventually, and hence
$
\P_{n_\alpha}(\{m\})=\frac{1}{n_\alpha}$
eventually. Therefore
$
\mu(\{m\})
=
\lim_\alpha \P_{n_\alpha}(\{m\})
=
0$.
If $\mu$ were countably additive, then
\[
1=\mu(\mathbb N)=\sum_{m=1}^\infty \mu(\{m\})=0,
\]
a contradiction. Thus $\mu\notin\Mone$. 

Lemma~\ref{lem:weak-star-jipr-zero-iff-equal}
yields
$
H(\mu\mid\mu)
=
0$.
Since
$\mu\in\cconv^*(\Pcal)\cap\cconv^*(\Qcal)$ and since relative entropy is nonnegative, we conclude that
\begin{equation} \label{eq:example-weak-star-value-zero}
\min_{\nu\in \cconv^*(\Qcal)}\min_{\lambda\in \cconv^*(\Pcal)}
H(\nu\mid\lambda)
=
0.
\end{equation}
Let $(\nu^*,\lambda^*)$ be now any minimizing pair in \eqref{eq:example-weak-star-value-zero}.
By Lemma~\ref{lem:weak-star-jipr-zero-iff-equal}, this implies
$\nu^*=\lambda^*$. Hence every minimizing pair is of the form
$(\lambda^*,\lambda^*)$, with
$
\lambda^*\in\cconv^*(\Pcal)\cap\cconv^*(\Qcal)$.
It remains to show that this intersection contains only purely finitely additive probability measures.

Indeed, consider any
$
\lambda\in\cconv^*(\Pcal)\cap\cconv^*(\Qcal)$.
For every $n,k\in\mathbb N$, we have
$
\P_n(\{k\})\ge \P_n(\{k+1\})$.
This inequality is preserved under finite convex combinations and under
weak-$*$ limits because evaluation at $\{k\}$ and $\{k+1\}$ is weak-$*$ continuous. Hence, for every $k\in\mathbb N$,
$
\lambda(\{k\})\ge \lambda(\{k+1\})$.
On the other hand, every $\Q_n$ satisfies $\Q_n(\{1\})=0$, and this property is
also preserved under finite convex combinations and weak-$*$ limits. Since
$\lambda\in\cconv^*(\Qcal)$, it follows that
$
\lambda(\{1\})=0$.
Therefore,
\[
0\le \lambda(\{k\})\le \lambda(\{1\})=0,
\qquad k\in\mathbb N,
\]
and hence $\lambda(\{k\})=0$ for every $k$. 
Let now
$\lambda=\lambda_c+\lambda_p$ be the Yosida--Hewitt decomposition of
$\lambda$ into its countably additive and purely finitely additive parts.
Since $\lambda_c\le\lambda$, we have $\lambda_c(\{k\})=0$ for every
$k$. Countable additivity of $\lambda_c$ then yields 
\[
\lambda_c(\mathbb N)=\sum_{k=1}^\infty \lambda_c(\{k\})=0.
\]
Hence $\lambda_c=0$, and therefore $\lambda$ is purely finitely additive.

\section{Details from Section~\ref{sec:grow}}\label{app:bounded-second-moment}

Example~\ref{eg:bounded-mean} presented a simple example of $\Pcal,\Qcal$ that are convex and weakly compact, and showed how to apply this paper's techniques to derive strong duality, along with the GROW e-variable. We now do the same for a slightly more sophisticated example, which involves mean shifts of unbounded random variables with bounded second moment.

\begin{example}[Mean shift with second-moment constraints and no domination]
Let $\Omega=\mathbb R$ and $Z(\omega)=\omega$, and define
\[
    \Pcal
    =
    \left\{
    \P\in \Mone:
    \E_\P[Z]\le0,\ \E_\P[Z^2]\le1
    \right\},
\]
and
\[
    \Qcal
    =
    \left\{
    \Q\in \Mone:
    \E_\Q[Z]\ge1,\ \E_\Q[Z^2]\le2
    \right\}.
\]
Both classes are convex and weakly compact. Indeed, the second-moment
bounds imply tightness, while the moment constraints are closed under
weak convergence because the second moments are uniformly bounded and because of the Portmanteau theorem.
The classes are not dominated by any common finite reference measure,
since $\Pcal$ contains $\delta_\omega$ for every $\omega\in[-1,0]$ and
$\Qcal$ contains $\delta_\omega$ for every $\omega\in[1,\sqrt2]$.

Let
\[
    \varphi=\frac{1+\sqrt5}{2},
    \qquad
    \omega_-=-\frac1\varphi,
    \qquad
    \omega_+=\varphi.
\]
Define
\[
    \P^*
    =
    \frac{\varphi^2}{\varphi^2+1}\delta_{\omega_-}
    +
    \frac{1}{\varphi^2+1}\delta_{\omega_+},
\]
and
\[
    \Q^*
    =
    \frac{1}{\varphi^2+1}\delta_{\omega_-}
    +
    \frac{\varphi^2}{\varphi^2+1}\delta_{\omega_+}.
\]
Then
\(
    \E_{\P^*}[Z]=0,
    \E_{\P^*}[Z^2]=1,
\)
so $\P^*\in\Pcal$, while
\(
    \E_{\Q^*}[Z]=1,
    \E_{\Q^*}[Z^2]=2,
\)
so $\Q^*\in\Qcal$.

On the support $\{\omega_-,\omega_+\}$, the likelihood ratio is
\[
    \frac{d\Q^*}{d\P^*}(\omega_-)=\frac1{\varphi^2},
    \qquad
    \frac{d\Q^*}{d\P^*}(\omega_+)=\varphi^2,
\]
yielding
$H(\Q^*\mid \P^*)=
    \frac{2\log\varphi}{\sqrt5}$.
We next choose a useful pointwise version of the likelihood ratio. 
\cite{larsson2026testing} indicate that all  e-variables can be dominated by one with a quadratic form, so below we choose the quadratic to match the value of $\frac{d\Q^*}{d\P^*}$ at $\omega_-$ and $\omega_+$.   
Let
\[
    A
    =
    \frac25\left(1+\frac{4\log\varphi}{\sqrt5}\right),
\]
and define
\[
    X^*(x)=A(1+x)+(1-A)x^2,
    \qquad x\in\mathbb R.
\]
Then $2/3<A<3/4$, and $X^*(x)> 0$ for all $x\in\mathbb R$. Moreover,
since $1+\varphi=\varphi^2$ and
$1-1/\varphi=1/\varphi^2$, we have
\[
    X^*(\omega_-)=\frac1{\varphi^2},
    \qquad
    X^*(\omega_+)=\varphi^2.
\]
Thus $X^*$ is a $\P^*$-version of $d\Q^*/d\P^*$, and 
\[
    \E_\P[X^*] = A (1+\E_\P[Z]) + (1-A) \E_\P[Z^2] \leq 1
\]
yields $X^* \in \Ecal$ and hence that $\P^*$ is the RIPr of $\Q^*$ onto $\Peff$ (by~\cite{larsson2025numeraire}).

We now show that Assumption~\ref{ass:composite-P} holds. First note that
\[
    \inf_{\RR\in\Qcal}H(\RR\mid \P^*)=H(\Q^*\mid \P^*).
\]
Indeed, since $\P^*$ is supported on $\{\omega_-,\omega_+\}$, any $\RR$ with
$H(\RR\mid \P^*)<\infty$ must also be supported on $\{\omega_-,\omega_+\}$. Writing
\[
    \RR =(1-q)\delta_{\omega_-}+q\delta_{\omega_+},
\]
the constraints $\E_\RR[Z]\ge1$ and $\E_\RR[Z^2]\le2$ force
\(
    q=\frac{\varphi^2}{\varphi^2+1},
\)
and hence $\RR=\Q^*$.  

It remains to check that
\(
    \E_\Q[\log X^*]\ge H(\Q^*\mid \P^*)
    \text{ for every }\Q\in\Qcal
\).
For the above choice of $A$, a direct one-dimensional calculus check shows
that with
\[
    c=\frac{A-2}{4} < 0, \qquad
    \lambda=4A-2>0,
    \qquad
    \beta=\frac32 A-1>0,
\]
we have
\[
    h(\omega)=c+\lambda \omega-\beta \omega^2
    \le
    \log X^*(\omega)
    \qquad
    \text{for all }\omega\in\mathbb R,
\]
with equality at $\omega_-$ and $\omega_+$. Therefore, for any $\Q\in\Qcal$,
\[
\begin{aligned}
    \E_\Q[\log X^*]
 \ge
    \E_\Q[h(Z)] =
    c+\lambda \E_\Q[Z]-\beta \E_\Q[Z^2] \ge
    c+\lambda-2\beta.
\end{aligned}
\]
Since $\Q^*$ has $\E_{\Q^*}[Z]=1$ and $\E_{\Q^*}[Z^2]=2$, and since
$h(Z)=\log(X^*)$ on the support of $\Q^*$,
\[
    c+\lambda-2\beta
    =
    \E_{\Q^*}[h(Z)]
    =
    \E_{\Q^*}[\log X^*]
    =
    H(\Q^*\mid \P^*).
\]
Thus
\[
    \E_\Q[\log X^*]\ge 
    \inf_{\RR\in\Qcal}H(\RR\mid \P^*)
    \qquad
    \text{for every }\Q\in\Qcal.
\]
This is Assumption~\ref{ass:composite-P}.
Consequently, the likelihood-ratio version $X^*$ witnesses strong duality:
\[
    \Grow 
    =
    \sup_{X\in\Ecal}
    \inf_{\Q\in\Qcal}\E_\Q[\log X]
    =
    H(\Q^*\mid \P^*)
    =
    \frac{2\log\varphi}{\sqrt5}.
\]
Clearly $X^*$ is also continuous, though not bounded, but Theorem~\ref{thm:nondominated-compact-convex} yields that $\Grow = \Grow_b = \Grow_{bb}  = \Grow_{bc}$ as well.
\end{example}

\section{Details from Section~\ref{sec:examples}}\label{app:sec5}

\begin{proof}[Derivation of \( \inf_{\Q\in\Qcal_a}H(\Q\mid \P)\) in Example~\ref{eg:csiszar}]
Let \( A_m=[m,\infty),  \alpha_m=\Q^*(A_m), \) and let $\RR_m$ be the conditional distribution of $\Q^*$ on $A_m$, i.e., \( \RR_m(\cdot)=\Q^*(\cdot\mid A_m). \) A direct calculation gives \( \alpha_m = \frac{1}{3(m+1)^2}+\frac{2}{3(m+1)^3}. \) 
Moreover,
\( M_m=\int s d \RR_m(s) \uparrow\infty. \) Indeed, 
\[
\frac{1}{\alpha_m} \int_m^\infty s\,q^*(s)\,ds = \frac2{3\alpha_m} \left( \frac{1}{m+1} + \frac{1}{(m+1)^2} - \frac{1}{(m+1)^3} \right)
\geq 2m. 
\] 
 For all large $m$, define \( \eta_m = \frac{a-\frac23}{M_m-\frac23} \in(0,1) \) and set \( \Q_m=(1-\eta_m)\Q^*+\eta_m \RR_m. \) 
        Then 
        \[ \int s d \Q_m(s) = (1-\eta_m)\frac23+\eta_m M_m = a, \]
        so $\Q_m\in\Qcal_a$. By convexity of relative entropy we get, 
        \[ 
        H(\Q_m\mid \Q^*) \le (1-\eta_m)H(\Q^*\mid \Q^*)+\eta_m H(\RR_m\mid \Q^*) = \eta_m\log\frac{1}{\alpha_m}
        \to 0. 
        \] 
        Therefore, 
        \[ 
        H(\Q_m\mid \P) = H(\Q_m\mid \Q^*)+ \E_{\Q_m}\left[\log\frac{d\Q^*}{d\P}\right] 
        = H(\Q_m\mid \Q^*)+ a+\log\frac23 \to a+\log\frac23. 
        \] 
        Hence \( \inf_{\Q\in\Qcal_a}H(\Q\mid \P) = a+\log\frac23.\)
\end{proof}

\begin{proof}[Proof of Theorem~\ref{T:shielding}]
For every $\Q\in\Qcal_\infty$, $\sup_{X \in \Ecal} \E_\Q[\log X] \geq \E_\Q[\log Y] = \infty$, and thus from~\eqref{eq:numeraire-duality}, 
\(
\inf_{\P\in\Peff}H(\Q\mid \P)=\infty.
\)
Therefore,
\[
\inf_{\Q\in\Qcal}
\inf_{\P\in\Peff}H(\Q \mid \P)
=
\min\left(
\inf_{\Q\in\Qcal_I}
\inf_{\P\in\Peff}H(\Q\mid \P),
\,
\inf_{\Q\in\Qcal_\infty}
\inf_{\P\in\Peff}H(\Q\mid \P)
\right)
=
I.
\]
Thus, we only have to prove the first equality in~\eqref{eq:shielded-strong-duality}, and indeed only \(
\sup_{X\in\Ecal}
\inf_{\Q\in\Qcal}\E_\Q[\log X]
\ge I,
\) needs to be argued. To this end, fix $\varepsilon>0$. By the definition of
$I$, there exists $X_\varepsilon\in\Ecal$ such that
\[
I - \varepsilon \leq \inf_{\Q\in\Qcal_I}\E_\Q[\log X_\varepsilon].
\]
For $\delta\in(0,1)$ define now
\[
X_{\varepsilon,\delta}
=
(1-\delta)X_\varepsilon+\delta Y \in \Ecal.
\]
If $\Q\in\Qcal_I$, then
\(
X_{\varepsilon,\delta}\ge (1-\delta)X_\varepsilon,
\)
and therefore
\[
\E_\Q[\log X_{\varepsilon,\delta}]
\ge 
\E_\Q[\log X_\varepsilon]+\log(1-\delta)
\ge
I-\varepsilon+\log(1-\delta).
\]
If $\Q\in\Qcal_\infty$, then
\(
X_{\varepsilon,\delta}\ge \delta Y,
\)
and hence, by \eqref{eq:shield-assumption},
\[
\E_\Q[\log X_{\varepsilon,\delta}]
\ge
\log\delta+\E_\Q[\log Y]
=
\infty.
\]
The last two displays together yield
\[
\inf_{\Q\in\Qcal}\E_\Q[\log X_{\varepsilon,\delta}]
\ge
I-\varepsilon+\log(1-\delta).
\]
Letting $\delta\downarrow0$ and $\varepsilon\downarrow0$ gives
\(
\sup_{X\in\Ecal}
\inf_{\Q\in\Qcal}\E_\Q[\log X]
\ge I,
\)
concluding the proof.
\end{proof}

\end{document}